\documentclass[a4paper]{amsart}
\usepackage{amssymb, enumerate}
\usepackage{xy} \xyoption{all}
\usepackage{aliascnt,hyperref}

\DeclareMathOperator{\id}{id}
\DeclareMathOperator{\ev}{ev}
\DeclareMathOperator{\im}{im}
\DeclareMathOperator{\Cov}{Cov}
\DeclareMathOperator{\gen}{gen}
\DeclareMathOperator{\dist}{dist}

\newcommand{\pt}{\text{pt}}
\newcommand{\cone}[1]{\mathcal{C}#1}

\newcommand{\approxARef}[1]{($\mathfrak{A}$#1)}
\newcommand{\myRef}[1]{(\ref{#1})}
\newcommand{\KK}{{\mathbb{K}}}
\newcommand{\NN}{{\mathbb{N}}}
\newcommand{\CC}{{\mathbb{C}}}
\newcommand{\otimesMax}{\otimes_\text{max}}
\newcommand{\ca}{{\mbox{$C^*$-al}\-ge\-bra}}
\newcommand{\ssubset}{\subset\!\!\subset}

\newtheorem{lma}{Lemma}[section]

\newaliascnt{thmCt}{lma}
\newtheorem{thm}[thmCt]{Theorem}
\aliascntresetthe{thmCt}

\newaliascnt{corCt}{lma}
\newtheorem{cor}[corCt]{Corollary}
\aliascntresetthe{corCt}

\newaliascnt{prpCt}{lma}
\newtheorem{prp}[prpCt]{Proposition}
\aliascntresetthe{prpCt}

\theoremstyle{definition}

\newaliascnt{pgrCt}{lma}
\newtheorem{pgr}[pgrCt]{}
\aliascntresetthe{pgrCt}

\newaliascnt{dfnCt}{lma}
\newtheorem{dfn}[dfnCt]{Definition}
\aliascntresetthe{dfnCt}

\newaliascnt{rmkCt}{lma}
\newtheorem{rmk}[rmkCt]{Remark}
\aliascntresetthe{rmkCt}

\newaliascnt{rmksCt}{lma}

\aliascntresetthe{rmksCt}

\newaliascnt{exaCt}{lma}
\newtheorem{exa}[exaCt]{Example}
\aliascntresetthe{exaCt}

\newaliascnt{qstCt}{lma}
\newtheorem{qst}[qstCt]{Question}
\aliascntresetthe{qstCt}

\title[Limits of projectives]{Inductive limits of projective $C^*$-algebras}%
\author{Hannes Thiel}
\address{Hannes Thiel
Mathematisches Institut, Fachbereich Mathematik und Informatik der
Universit\"at M\"unster, Einsteinstrasse 62, 48149 M\"unster, Germany.}
\email{hannes.thiel@uni-muenster.de}
\urladdr{www.math.uni-muenster.de/u/hannes.thiel/}

\thanks{
The author was partially supported by the Marie Curie Research Training Network EU-NCG, by the Danish National Research Foundation through the Centre for Symmetry and Deformation, and by the Deutsche Forschungsgemeinschaft (SFB 878 Groups, Geometry \& Actions).}

\subjclass[2010]%
{Primary
46L05, 
46M10; 
Secondary
46L85, 
46M20, 
54C55, 
54C56, 
55M15, 
55P55
}

\date{\today}

\begin{document}

\begin{abstract}
We show that a separable \ca{} is an inductive limits of projective \ca{s} if and only if it has trivial shape, that is, if it is shape equivalent to the zero \ca{}.
In particular, every contractible \ca{} is an inductive limit of projectives, and one may assume that the connecting morphisms are surjective.
Interestingly, an example of Dadarlat shows that trivial shape does not pass to full hereditary sub-\ca{s}.
It then follows that the same fails for projectivity.

To obtain these results, we develop criteria for inductive limit decompositions, and we discuss the relation with different concepts of approximation.

As a main application of our findings we show that a \ca{} is (weakly) projective if and only if it is (weakly) semiprojective and has trivial shape.
It follows that a \ca{} is projective if and only if it is contractible and semiprojective.
This confirms a conjecture of Loring.
\end{abstract}

\maketitle

\section{Introduction}

Shape theory and homotopy theory are tools to study global properties of spaces.
However, homotopy theory gives useful results mainly for spaces with good local behavior (that is, without singularities).
For such well-behaved spaces both theories agree, and one usually employs homotopy theory which is easier to compute.
To study more general spaces with possible singularities, one uses shape theory.
The idea is to abstract from the local behavior of a space, and focus on its global behavior, its `shape'.

One way of doing this, is to approximate a space by nicer spaces, the building blocks.
In the commutative world the building blocks are the so-called absolute neighborhood retracts, abbreviated by ANR.
The approximation is organized in an inverse limit structure, and instead of looking at the original space one studies an associated inverse system of ANRs.

After shape theory was successfully used to study spaces, it was introduced to the study of noncommutative spaces (that is, \ca{s}) by Effros and Kaminker, \cite{EffKam86ShapeThy}, and shortly after developed to its modern form by Blackadar, \cite{Bla85ShapeThy}.
Shape theory works best when restricted to metrizable spaces, and similarly for noncommutative shape theory one restricts attention to separable \ca{s}.

The building blocks of noncommutative shape theory are the semiprojective \ca{s}, which are defined in analogy to ANRs.
Since the category of commutative \ca{s} is dual to the category of spaces, the approximation by an inverse system for spaces is turned into an approximation by an inductive system for \ca{s}.
Then, approximating a \ca{} by `nice' \ca{s} means to write it as an inductive limit of semiprojective \ca{s}.

This raises the natural question of whether there are enough building blocks to approximate every space.
This is true in the commutative world, as every metric space is an inverse limit of ANRs.
The analog for \ca{s} is still an open problem, first asked by Blackadar:

\begin{qst}[{Blackadar, \cite[4.4]{Bla85ShapeThy}}]
\label{quest:S01:SSS}
Is every separable \ca{} an inductive limit of semiprojective \ca{s}?
\end{qst}

In this paper we study the related question of which \ca{s} are inductive limits of projective \ca{s}.
A necessary condition is that such a \ca{} has trivial shape, that is, it is shape equivalent to the zero \ca{}, since this holds for projective \ca{s} and is preserved by inductive limits.
We will show that the converse is also true, that is, a separable \ca{} is an inductive limit of projective \ca{s} if and only if it has trivial shape;
see \autoref{prop:S04:TFAE-trivSh}.
This also gives a positive answer to \autoref{quest:S01:SSS} for \ca{s} with trivial shape, a class which is quite large since it contains for instance all contractible \ca{s}.

This paper is organized as follows.
In \autoref{sect:S02:preliminaries} we recall the basic definitions.
Then, in \autoref{sect:S03:approximation} we discuss different concepts of how a \ca{} can be `approximated' by other \ca{s}, for instance as an inductive limit.
If $\mathcal{C}$ is a class of \ca{s}, then an inductive limit of algebras in $\mathcal{C}$ is called an $A\mathcal{C}$-algebra.
We suggest to use the formulation that $A$ is `$\mathcal{C}$-like' if it can be approximated by sub-\ca{s} from $\mathcal{C}$, see \autoref{defn:S03:P-like} and \autoref{prop:S03:P_like_is_right_notion}.

Building on a one-sided approximate intertwining argument, due to Elliott in \cite[2.1, 2.3]{Ell93ClassRR0}, see \autoref{prop:S03:Intertwining}, we give two criteria to show that a given \ca{} is an $A\mathcal{C}$-algebra.
We assume that the class $\mathcal{C}$ of building blocks consists of weakly semiprojective \ca{s}.
Then every separable $A\mathcal{C}$-like \ca{} is already an $A\mathcal{C}$-algebra, see \autoref{prop:S03:AC-like_implies_AC}, and every $AA\mathcal{C}$-algebra is already an $A\mathcal{C}$-algebra, see \autoref{prop:S03:AAC_implies_AC}.

In \autoref{sect:S04:trivial_shape} we study the class of \ca{s} with trivial shape.
We show that these are exactly the \ca{s} that are inductive limits of projective \ca{s}, see \autoref{prop:S04:TFAE-trivSh}.
Moreover, one may assume that the connecting morphisms are surjective, since we show in \autoref{prop:S04:change_connecting_mor_surj} that every inductive system can be changed so that the connecting morphisms become surjective while the limit is unchanged.

As a corollary, we obtain that every separable, contractible \ca{} is an inductive limit of projective \ca{s}, see \autoref{prop:S04:contr_AP}.
We discuss permanence properties of trivial shape, see \autoref{prop:S04:Permanence_trivSh}.
It follows from an example of Dadarlat that trivial shape does not pass to full hereditary sub-\ca{s}, see \autoref{pargr:S04:exmpl_Dadarlat}.
We deduce that also projectivity does not pass to full hereditary sub-\ca{s}, see \autoref{prop:S04:Projectivity_not_hereditary}.

In \autoref{sect:S05:relations_classes} we show some non-commutative analogs of results in commutative shape theory.
We prove that a separable \ca{} is (weakly) projective if and only if is is  (weakly) semiprojective and has trivial shape, see \autoref{prop:S05:wP_iff_wSP_trivSh}.
It follows that a separable \ca{} is projective if and only if it is is semiprojective and contractible, see \autoref{prop:S05:P_iff_SP-C}.
This confirms a conjecture of Loring.

\section{Preliminaries}
\label{sect:S02:preliminaries}

By a morphism between \ca{s} we mean a ${}^\ast$-homomorphism.
Given a morphism $\varphi$, we use $\im(\varphi)$ to denote its image, and we use $\ker(\varphi)$ to denote its kernel.
Ideals are understood to be closed and two-sided ideals.
We use the symbol `$\simeq$' to denote homotopy equivalence, both for objects and morphisms.

We use the following notations.
For $\varepsilon>0$, a subset $F$ of a \ca{} $A$ is said to be \emph{$\varepsilon$-contained} in another subset $G$, denoted by $F\subseteq_\varepsilon G$, if for every $x\in F$ there exists some $y\in G$ such that $\|x-y\|<\varepsilon$.
Given two morphisms $\varphi,\psi\colon A\to B$ between \ca{s} and a subset $F\subseteq A$ we say that \emph{$\varphi$ and $\psi$ agree on $F$}, denoted $\varphi=^F\psi$, if $\varphi(x)=\psi(x)$ for all $x\in F$.
If, moreover, $\varepsilon>0$ is given, then we say that \emph{$\varphi$ and $\psi$ agree on $F$ up to $\varepsilon$}, denoted $\varphi=_\varepsilon^F\psi$, if $\|\varphi(x)-\psi(x)\|<\varepsilon$ for all $x\in F$.
We write $F\ssubset A$ to denote that $F$ is a finite subset of $A$.

\begin{pgr}
\label{pargr:S02:wSP}
We consider shape theory for separable \ca{} in the sense of Blackadar, see \cite{Bla85ShapeThy}.
In this paragraph, which is a shortened version of \cite[2.2]{SoeThi12CharCommutSP}, we recall the main notions.

A morphism $\varphi\colon A\to B$ is called \emph{(weakly) projective} if whenever $C$ is a \ca{}, $J\lhd A$ is an ideal and $\sigma\colon B\to C/J$ is a morphism (and $\varepsilon>0$ and $F\ssubset A$), there exists a morphism $\psi\colon A\to C$ such that $\pi\circ\psi=\sigma\circ\varphi$ (resp. $\pi\circ\psi=_\varepsilon^F\sigma\circ\varphi$), where $\pi\colon C\to C/J$ is the quotient morphism.
This means that the left diagram below can be completed to commute (up to $\varepsilon$ on $F$).

A \ca{} $A$ is called \emph{(weakly) projective} if the identity morphism $\id_A\colon A\to A$ is (weakly) projective.

\begin{figure}[h]
\centering
\begin{minipage}{.5\textwidth}
\centering
\makebox{
\xymatrix{
& & C \ar[d]^{\pi} \\
A \ar[r]_{\varphi} \ar@{..>}[urr]^{\psi} & B \ar[r]_{\sigma}& C/J
}}
\end{minipage}%
\begin{minipage}{.5\textwidth}
\centering
\makebox{
\xymatrix{
& & C \ar[d] \\
& & C/J_k \ar[d]^{\pi_k} \\
A \ar[r]_{\varphi} \ar@{..>}[urr]^{\psi} & B \ar[r]_{\sigma}  & C/\overline{\bigcup_k J_k}
}}
\end{minipage}
\end{figure}

A morphism $\varphi\colon A\to B$ is called \emph{(weakly) semiprojective} if whenever $C$ is a \ca{}, $J_1\subseteq J_2\subseteq\cdots$ is an increasing sequence of ideals of $C$ and $\sigma\colon B\to C/\overline{\bigcup_k J_k}$ is a morphism (and $\varepsilon>0$ and $F\ssubset A$), there exist an index $k$ and a morphism $\psi\colon A\to C/J_k$ such that $\pi_k\circ\psi=\sigma\circ\varphi$ (resp. $\pi_k\circ\psi=_\varepsilon^F\sigma\circ\varphi$), where $\pi_k\colon C/J_k\to C/\overline{\bigcup_k J_k}$ is the quotient morphism.
This means that the right diagram above can be completed to commute (up to $\varepsilon$ on $F$).

A \ca{} $A$ is said to be \emph{(weakly) semiprojective} if the identity morphism $\id_A\colon A\to A$ is (weakly) semiprojective.
\end{pgr}

\begin{pgr}
\label{pargr:S02:shape_system}
A (sequential) \emph{inductive system} of \ca{s} is a collection $\mathcal{A}=(A_k,\gamma_k)$ of \ca{s} $A_1,A_2,\ldots$ together with morphisms $\gamma_k\colon A_k\to A_{k+1}$ for each $k$.
If $k<l$, then we let $\gamma_{l,k}:=\gamma_{l-1}\circ\cdots\circ\gamma_{k+1}\circ\gamma_k\colon A_k\to A_l$ denote the composition of connecting morphisms.
By $\varinjlim\mathcal{A}$ or $\varinjlim A_k$ we denote the inductive limit of an inductive system, and by $\gamma_{\infty,k}\colon A_k\to\varinjlim A_k$ we denote the canonical morphism into the inductive limit.

Let $A$ be a separable \ca{}.
A \emph{shape system} for $A$ is an inductive system $(A_k,\gamma_k)$ such that $A\cong\varinjlim A_k$ and such that the connecting morphisms $\gamma_k\colon A_k\to A_{k+1}$ are semiprojective.
By \cite[Theorem~4.3]{Bla85ShapeThy}, every separable \ca{} has a shape system consisting of finitely generated \ca{s}.

Two inductive systems $\mathcal{A}=(A_k,\gamma_k)$ and $\mathcal{B}=(B_n,\theta_n)$ are called \emph{(shape) equivalent}, denoted $\mathcal{A}\sim\mathcal{B}$, if there exist an increasing sequences of indices $k_1<n_1<k_2<n_2<\ldots$ and morphisms $\alpha_i\colon A_{k_i}\to B_{n_i}$ and $\beta_i\colon B_{n_i}\to A_{k_{i+1}}$ such that $\beta_i\circ\alpha_i\simeq\gamma_{k_{i+1},k_i}$ and $\alpha_{i+1}\circ\beta_i\simeq\theta_{n_{i+1},n_i}$ for all $i$.
The situation is shown in the following diagram which commutes up to homotopy.
\begin{center}
\makebox{
\xymatrix{
A_{k_1} \ar[rr]^{\gamma_{k_2,k_1}} \ar[dr]_{\alpha_1}
    & & A_{k_2} \ar[rr]^{\gamma_{k_3,k_2}} \ar[dr]_{\alpha_2}
    & & A_{k_3} \ar[rr] \ar[dr]_{\alpha_3}
    & & \ldots \ar[r]
    & A
    \\
& B_{n_1} \ar[rr]_{\theta_{n_2,n_1}} \ar[ur]^{\beta_1}
    & & B_{n_2} \ar[rr]_{\theta_{n_3,n_2}} \ar[ur]^{\beta_2}
    & & B_{n_3} \ar[r]
    & \ldots \ar[r]
    &  B
\\
}}
\end{center}

If we have $\alpha_i,\beta_i$ as above with only $\beta_i\circ\alpha_i\simeq\gamma_{k_{i+1},k_i}$ for all $i$, then we say that $\mathcal{A}$ is \emph{(shape) dominated} by $\mathcal{B}$, denoted $\mathcal{A}\precsim\mathcal{B}$.
Of course $\mathcal{A}\sim\mathcal{B}$ implies $\mathcal{A}\precsim\mathcal{B}$ and $\mathcal{B}\precsim\mathcal{A}$, but the converse is false.
Nevertheless $\sim$ is an equivalence relation, and $\precsim$ is transitive.

Any two shape systems of a \ca{} are equivalent.
Given two \ca{s} $A$ and $B$ we say that $A$ is \emph{shape equivalent} to $B$, denoted $A\sim_{Sh}B$, if they have some shape systems that are equivalent.
We say $A$ is \emph{shape dominated} by $B$, denoted $A\precsim_{Sh}B$, if some shape system of $A$ is dominated by some shape system of $B$.

Shape is coarser than homotopy in the following sense: If $A$ and $B$ are homotopy equivalent (denoted $A\simeq B$), then $A\sim_{Sh}B$.
Moreover, if $A$ is homotopy dominated by $B$, then $A\precsim_{Sh}B$.
\end{pgr}

\begin{thm}[{Effros, Kaminker, \cite[3.2]{EffKam86ShapeThy}, also Blackadar, \cite[Theorem~3.1, 3.3]{Bla85ShapeThy}}]
\label{prop:S02:SP_lifting_limits}
Let $\varphi\colon A\to B$ be a semiprojective morphism, and let $(C_k,\gamma_k)$ be an inductive system with limit $C$.
Then:
\begin{enumerate}
\item
Let $\sigma\colon B\to C$ be a morphism.
Then for $k$ large enough there exist morphisms $\psi_k\colon A\to C_k$ such that $\gamma_{\infty, k}\circ\psi_k\simeq\sigma\circ\varphi$ and such that $\gamma_{\infty,k}\circ\psi_k$ converges pointwise to $\sigma\circ\varphi$.
This means that the left diagram below can be completed to commute up to homotopy.
\item
Let $\sigma_1,\sigma_2\colon B\to C_k$ be two morphisms with $\gamma_{\infty,k}\circ\sigma_1\simeq\gamma_{\infty,k}\circ\sigma_2$.
Then for $n\geq k$ large enough, already the morphisms $\gamma_{n,k}\circ\sigma_1\circ\varphi$ and $\gamma_{n,k}\circ\sigma_2\circ\varphi$ are homotopic.
The situation is shown in the right diagram below.
\end{enumerate}
\begin{center}
\makebox{
\xymatrix{
A \ar[r]^{\varphi} \ar@{..>}[d]_{\psi_k}
& B \ar[d]^{\sigma}
& &
A \ar[r]^{\varphi}
& B \ar@<2pt>[dl]^{\sigma_1} \ar@<-2pt>[dl]_{\sigma_2}
\\
C_k \ar[r]_{\gamma_{\infty, k}} & C
& &
C_k \ar[r]_{\gamma_{n,k}}
& C_n \ar[r]
& C
}}
\end{center}
\end{thm}

\begin{rmk}
Let $(C_k,\gamma_k)$ be an inductive system with limit $C$.
We use $[A,C_k]$ to denote the set of homotopy classes of morphisms from $A$ to $C_k$.
The connecting morphism $\gamma_k\colon C_k\to C_{k+1}$ induces a map $(\gamma_k)_\ast\colon [A,C_k]\to[A,C_{k+1}]$, and the morphism $\gamma_{\infty,k}\colon C_k\to C$ induces a map $(\gamma_{\infty,k})_\ast\colon[A,C_k]\to[A,C]$.

Note that $(\gamma_{\infty,k})_\ast=(\gamma_{\infty,k+1})_\ast\circ(\gamma_k)_\ast$, so that we get a natural map
\begin{align*}
\Phi\colon\varinjlim [A,C_k] \to [A,\varinjlim C_k]=[A,C].
\end{align*}

Statement $(1)$ of the above \autoref{prop:S02:SP_lifting_limits} means that $\Phi$ is surjective, while statement $(2)$ means exactly that $\Phi$ is injective.
\end{rmk}

\autoref{prop:S02:SP_lifting_limits} is proved using a so-called mapping telescope construction, due to L.G.~Brown.
The same proof gives the following partial analog of the above result for weakly semiprojective morphisms:

\begin{prp}
\label{prop:S02:WSP_lifting_limits}
Let $\varphi\colon A\to B$ be a weakly semiprojective morphism, let $(C_k,\gamma_k)$ be an inductive system, and let $\sigma\colon B\to \varinjlim C_k$ be a morphism.
Then, for every $\varepsilon>0$ and $F\ssubset A$, there exist an index $k$ and a morphism $\psi\colon A\to C_k$ such that $\gamma_{\infty,k}\circ\psi=^F_\varepsilon\sigma\circ\varphi$.
\end{prp}

\begin{rmk}
The definition of (weak) semiprojectivity (\autoref{pargr:S02:wSP}) may be reformulated as follows:

A morphism $\varphi\colon A\to B$ is (weakly) semiprojective if for every inductive system $(D_k,\gamma_k)$ \emph{with surjective connecting morphisms}, and for every morphism $\sigma\colon B\to \varinjlim D_k$ (and $\varepsilon>0$ and $F\ssubset A$), there exist an index $k$ and a morphism $\psi\colon A\to D_k$ such that $\gamma_{\infty,k}\circ\psi=\sigma\circ\varphi$ (resp. $\gamma_{\infty,k}\circ\psi=^F_\varepsilon\sigma\circ\varphi$).

Thus, for the definition of (weak) semiprojectivity, we consider morphisms into the limit of an inductive system with surjective connecting morphisms, and we ask for approximate lifts.
It follows from \autoref{prop:S02:WSP_lifting_limits} that for weak semiprojectivity (but not for semiprojectivity) one may drop the condition that the connecting morphisms of the inductive system are surjective.
\end{rmk}

\begin{pgr}
\label{pargr:S02:Generators}
The \emph{generating rank} for a \ca{} $A$, denoted by $\gen(A)$, is the smallest number $n\in\{1,2,3,\ldots,\infty\}$ such that $A$ is generated (as a \ca{}) by $n$ \emph{self-adjoint} elements.
For more details we refer the reader to Nagisa, \cite{Nag04SingleGen}.

We say that $A$ is finitely generated if $\gen(A)<\infty$.
To define when a \ca{} is finitely presentation, one needs a theory of universal \ca{s} defined by generators and relations.
Depending on which relations one considers, one gets different notions of finite presentability.
In \cite{Bla85ShapeThy}, for instance, only polynomial relations are considered.
With this definition, not every finitely generated \ca{} is finitely presented.

More generally, one can define a relation to be an element of the universal \ca{} generated by a countable number of contractions.
This definition is used in \cite{Lor97LiftingSolutions}, and it is flexible enough to show that every finitely generated \ca{} is already finitely presented, see \cite[Lemma~2.2.5]{EilLorPed98StabAnticommutRel}.

Thus, in the results of \cite{Lor97LiftingSolutions} we may replace the assumption of finite presentation by finite generation, for example in \cite[Lemma~15.2.1, 15.2.2, p.118f]{Lor97LiftingSolutions}.
This can be improved even further, as was shown to the author by Chigogidze and Loring (private communication):
One may give a version of \cite[Lemma~15.2.1, p.118]{Lor97LiftingSolutions} which does not require the \ca{} to be finitely generated;
see \autoref{prop:S03:twist_maps_into_subalg}.
It follows that \cite[Lemma~15.2.2, p.119]{Lor97LiftingSolutions} remains true if one drops the assumption of finite generation (or presentation) completely;
see \autoref{prop:S03:C_like_implies_AC}.
\end{pgr}

\section{Approximation and criteria for inductive limits}
\label{sect:S03:approximation}

In this section we give criteria that allow one to write a \ca{} $A$ as an inductive limit of other \ca{s} that approximate $A$ in a nice way.
We start by reviewing the various ways that a \ca{} can be `approximated' by other \ca{s};
see \autoref{pargr:S03:Approximation}.
If $\mathcal{C}$ is a class of \ca{s}, then an inductive limit of algebras in $\mathcal{C}$ is called an $A\mathcal{C}$-algebra.
We suggest to use the formulation that $A$ is `$\mathcal{C}$-like' if it can be approximated by sub-\ca{s} from the class $\mathcal{C}$;
see \autoref{defn:S03:P-like} and \autoref{prop:S03:P_like_is_right_notion}.

As a basic tool to construct an inductive limit decomposition we use one-sided approximate intertwinings;
see \autoref{prop:S03:Intertwining}.
These were introduced by Elliott in \cite[2.1, 2.3]{Ell93ClassRR0} and they turned out to be very important in the classification of \ca{s}, see for example chapter $2.3$ of R{\o}rdam's book, \cite{Ror02Classification}.

Assuming that the class $\mathcal{C}$ consists of separable, weakly semiprojective \ca{s}, we deduce other criteria to write a separable \ca{} as an inductive limit of building blocks in $\mathcal{C}$.
In particular, we show that every separable $A\mathcal{C}$-like \ca{} is an $A\mathcal{C}$-algebras (\autoref{prop:S03:AC-like_implies_AC}), and every separable $AA\mathcal{C}$-algebra is already an $A\mathcal{C}$-algebra (\autoref{prop:S03:AAC_implies_AC}).
The latter statement gives a criterion when an `inductive limit of inductive limits is an inductive limit'.

For example, let $\mathcal{C}$ be the class of finite direct sums of matrices over the circle algebra $C(\mathbb{T})$.
Then the mentioned result means that an inductive limit of $A\mathbb{T}$-algebras is itself an $A\mathbb{T}$-algebra.
This is a well-known result, see for example \cite[Proposition 2]{LinRor95ExtLimitCircle} which is based on \cite[Theorem 4.3]{Ell93ClassRR0}.

\begin{pgr}
\label{pargr:S03:Approximation}
The term `approximation' is used in various contexts.
For instance, if $\mathcal{P}$ is some property that \ca{s} might enjoy, then a \ca{} is usually called \emph{approximately $\mathcal{P}$}, or an $A\mathcal{P}$-algebra, if it can be written as an inductive limit of \ca{s} with property $\mathcal{P}$.
In this sense one speaks of `approximately homogeneous' and `approximately subhomogeneous' \ca{s}.

If a separable \ca{} $A$ can be written as an inductive limit, $\varinjlim A_i$, of separable \ca{s}, indexed over a directed set $I$, then one may find a countable subset of ordered indices $i_1\leq i_2\leq\ldots$ such that $A$ is naturally an inductive limit of the sequential system $(A_{i_k})_k$.
In the study of noncommutative shape theory one usually restricts to separable \ca{s}.
We will therefore assume throughout that an $A\mathcal{P}$-algebra is an inductive limit of a sequential inductive system.
For some approximation results of non-separable \ca{s} we refer the reader to \cite{FarKat10NonsepUHF1}.

Another concept is approximation by subalgebras.
Given a \ca{} $A$, a family $\mathcal{B}$ of sub-\ca{s} is said to \emph{approximate} $A$
if for every $\varepsilon>0$ and $F\ssubset A$ there exists some algebra $B\in\mathcal{B}$ such that $F\subseteq_\varepsilon B$.
In the literature, this is often phrased as `$\mathcal{B}$ locally approximates $A$'.
Similarly, if $\mathcal{P}$ is some property of \ca{s}, then a \ca{} $A$ that can be approximated by sub-\ca{s} with property $\mathcal{P}$ is sometimes called `locally $\mathcal{P}$'.
In this sense one speaks of `locally (sub)homogeneous' \ca{s}.

However, sometimes the word `local' might lead to confusion:
Consider for instance the property of being contractible.
We show in \autoref{prop:S04:contr-like_implies_trivSh} that a \ca{} has trivial shape if it is approximated by contractible sub-\ca{s}.
One could phrase this as `locally contractible \ca{s} have trivial shape', but this would be counterintuitive to the terminology used for topological spaces.

The confusion is due to the contravariant duality between spaces and \ca{s}.
If we consider for instance a commutative \ca{} $C(X)$, then the elements $f\in C(X)$ are almost constant around each point $x\in X$.
Therefore, an approximation of $C(X)$ by sub-\ca{s} does not capture the local structure of $X$, it rather captures the global structure of $X$, its shape.
To prevent confusion, we will use the following definition:
\end{pgr}

\begin{dfn}
\label{defn:S03:P-like}
If $\mathcal{P}$ is some property of \ca{s}, then a \ca{} is said to be \emph{$\mathcal{P}$-like} if it can be approximated by sub-\ca{s} with property $\mathcal{P}$.
\end{dfn}

\begin{rmk}
\label{pargr:S03:P-likeness}
Using the above definition, \autoref{prop:S04:contr-like_implies_trivSh} reads as: `A contractible-like \ca{} has trivial shape'.
This might sound cumbersome, but it is motivated by the concept of $\mathcal{P}$-likeness for spaces, as defined by Mardesic and J. Segal, \cite[Definition 1]{MarSeg63EpsOntoPolyhedra}, and further developed by Marde\v{s}i\'c and Matijevi\'c, \cite{MarMat92PLike}.
In \autoref{prop:S03:P_like_is_right_notion} we show that for commutative \ca{s} both concepts agree.

For a space $X$, we let $\Cov(X)$ denote the collection of finite, open covers of $X$.
Given $\mathcal{U}_1,\mathcal{U}_1\in\Cov(X)$, we write $\mathcal{U}_1\leq\mathcal{U}_2$ if the cover $\mathcal{U}_1$ refines the cover $\mathcal{U}_2$, that is, if for every $U\in\mathcal{U}_1$ there exists some $U'\in\mathcal{U}_2$ such that $U\subseteq U'$.
We refer the reader to chapter $2$ of Nagami's book \cite{Nag70DimThy} for details.

We are working in the category of pointed spaces and pointed maps since it is the natural setting to study non-unital commutative \ca{s}, as pointed out in
\cite[II.2.2.7, p.61]{Bla06OpAlgs}.
If we include basepoints and restrict to compact spaces, then the definition of $\mathcal{P}$-likeness from \cite[Definition 1.2]{MarMat92PLike} becomes:
Let $\mathcal{P}$ be a non-empty class of pointed, compact, Hausdorff spaces.
A pointed, compact, Hausdorff space $X$ is said to be $\mathcal{P}$-like if for every $\mathcal{U}\in\Cov(X)$ there exists a pointed map $f\colon X\to Y$ onto some $Y\in\mathcal{P}$ and $\mathcal{V}\in\Cov(Y)$ such that $f^{-1}(\mathcal{V})\leq\mathcal{U}$, where $f^{-1}(\mathcal{V})=\{f^{-1}(V) : V\in\mathcal{V}\}$.

If, moreover, $X$ and all space in $\mathcal{P}$ are metric spaces, then one can show that $X$ is $\mathcal{P}$-like if and only if for every $\varepsilon>0$ there exists a pointed map $f\colon X\to Y$ onto some $Y\in\mathcal{P}$ such that the sets $f^{-1}(y)$ have diameter less than $\varepsilon$ (for all $y\in Y$).
This equivalent formulation is the original definition of $\mathcal{P}$-likeness for compact, metric spaces, \cite[Definition 1]{MarSeg63EpsOntoPolyhedra}.

Note that we have used $\mathcal{P}$ to denote both a class of spaces and a property that spaces might enjoy.
These are just different viewpoints, as we can naturally assign to a property the class of spaces with that property, and vice versa to each class of spaces the property of lying in that class.
\end{rmk}

Let us use the following notation for the next result:
If $X=(X,x_\infty)$ is a pointed space, then $C_0(X)=\{a\colon X\to\CC : a(x_\infty)=0\}$ denotes the \ca{} of continuous functions on $X$ vanishing at the basepoint.

\begin{prp}
\label{prop:S03:P_like_is_right_notion}
Let $X=(X,x_\infty)$ be a pointed, compact, Hausdorff space, and let $\mathcal{P}$ be a class of pointed, compact, Hausdorff spaces.
Then the following are equivalent:
\begin{enumerate}
\item
$X$ is $\mathcal{P}$-like.
\item
$C_0(X)$ can be approximated by sub-\ca{s} $C_0(Y)$ with $Y\in\mathcal{P}$.
\end{enumerate}
\end{prp}
\begin{proof}
To show that~(1) implies~(2), let $\varepsilon>0$ and $F\ssubset C_0(X)$.
Since $X$ is compact, there exists a cover $\mathcal{U}\in\Cov(X)$ such that $\|a(x)-a(x')\|<\varepsilon$ whenever $a\in F$ and $x,x'$ lie in some set $U\in\mathcal{U}$.
By assumption, there is a pointed map $f\colon X\to Y$ onto some space $Y=(Y,y_\infty)\in\mathcal{P}$ and $\mathcal{V}\in\Cov(Y)$ such that $f^{-1}(\mathcal{V})\leq\mathcal{U}$.
Note that $f$ induces an inclusion $f^\ast\colon C_0(Y)\to C_0(X)$.

Choose a partition of unity $(e_V)_{V\in\mathcal{V}}$ on $Y$ that is subordinate to $\mathcal{V}$.
For each $V\in\mathcal{V}$, choose a point $x_V\in f^{-1}(V)$ such that $x_V=x_\infty$ if $y_\infty\in V$.
Given $a\in F$, let us show that
\[
a\in_\varepsilon f^\ast(C_0(Y)).
\]
Set
\[
b:=\sum_V a(x_V)e_V.
\]
Note that $b(y_\infty)=0$ since $a(x_V)=0$ whenever $e_V(y_\infty)\neq 0$.
Given $x\in X$, we have $e_V(f(x))\neq 0$ only if $f(x)\in V$, in which case we have $x,x_V\in f^{-1}(V)$, and using that $f^{-1}(V)$ is contained in some set $U\in\mathcal{U}$, we deduce that $\|a(x)-a(x_V)\|<\varepsilon$.
Using this at the third step, we obtain that
\begin{align*}
\left\| a(x)-f^\ast(b)(x) \right\|
&= \left\| a(x)-\sum_V a(x_V)e_V(f(x)) \right\|
= \left\| \sum_V (a(x)-a(x_V))e_V(f(x)) \right\| \\
&< \varepsilon\cdot \left\| \sum_V e_V(f(x)) \right\|
=\varepsilon,
\end{align*}
for every $x\in X$.
Hence, $a\in_\varepsilon f^\ast(C_0(Y))$, as desired.

To show that~(2) implies~(1), let $\mathcal{U}=(U_\alpha)_\alpha\in\Cov(X)$ be a finite, open cover of $X$.
We need to find a space $Y\in\mathcal{P}$ together with a pointed, surjective map $f\colon X\to Y$ and $\mathcal{V}\in\Cov(V)$ such that $f^{-1}(\mathcal{V})\leq\mathcal{U}$.

By passing to a refinement, we may assume that $x_\infty$ is contained in just one $U_\alpha$, call it $U_\infty$.
Since $X$ is a normal space, for each $\alpha$ we may find an open set $V_\alpha\subseteq X$ such that $V_\alpha\subseteq\overline{V_\alpha}\subseteq U_\alpha$, and such that $(V_\alpha)_\alpha$ is a cover of $X$.
By Urysohn's lemma, there are continuous functions $a_\alpha\colon X\to\CC$ that are $1$ on $\overline{V_\alpha}$ and zero on $X\setminus U_\alpha$.
Note that $a_\alpha$ vanishes on $x_\infty$ for $\alpha\neq\infty$, so that $a_\alpha\in C_0(X)$ for $\alpha\neq\infty$.

By assumption, we obtain a sub-\ca{} $C_0(Y)$ of $C_0(X)$ that contains the $a_\alpha$ (for $\alpha\neq\infty$) up to $1/2$ and such that $Y=(Y,y_\infty)\in\mathcal{P}$.
The embedding corresponds to a pointed, surjective map $f\colon X\to Y$.
For $\alpha\neq\infty$, let $b_\alpha\in C_0(Y)$ be an element such that $\|a_\alpha-f^\ast(b_\alpha)\|<1/2$.

For $\alpha\neq\infty$, set
\[
W_\alpha := \big\{ y\in Y : \|b_\alpha(y)\|>1/2 \big\}.
\]
Further, set
\[
W_\infty := Y\setminus f\big( \bigcup_{\alpha\neq\infty}\overline{V_\alpha} \big).
\]
For $\alpha\neq\infty$, we compute
\[
f^{-1}(W_\alpha)
= \big\{ x\in X : \|b_\alpha(f(x))\|>1/2 \big\}
\subseteq \big\{ x\in X : \|a_\alpha(x)\|>0 \big\}
\subseteq U_\alpha,
\]
and
\[
f^{-1}(W_\alpha)
\supseteq \big\{ x\in X : \|a_\alpha(x)\|\geq 1 \big\}
\supseteq \overline{V_\alpha}.
\]
We also have
\[
f^{-1}(W_\infty)
\subseteq X\setminus \bigcup_{\alpha\neq\infty}\overline{V_\alpha}
\subseteq U_\infty.
\]
It follows that $f^{-1}(\bigcup_{\alpha\neq\infty}W_\alpha)\supseteq \bigcup_{\alpha\neq\infty}\overline{V_\alpha}$, and thus
\[
\bigcup_{\alpha\neq\infty}W_\alpha
\supseteq f \big( \bigcup_{\alpha\neq\infty}\overline{V_\alpha} \big)
= Y\setminus W_\infty.
\]
This shows that $\mathcal{W}:=(W_\alpha)_\alpha$ is a cover of $Y$ with $f^{-1}(\mathcal{W})\leq\mathcal{U}$, as desired.
\end{proof}

The following result formalizes the construction of a (special) one-sided approximate intertwining.
The idea goes back to Elliott, \cite[2.3,2.4]{Ell93ClassRR0}, see \cite[2.3]{Ror02Classification}.
It seems that the version given here has not appeared in the literature so far.
Note that we do not require any ordering on the index set of approximating algebras.

\begin{prp}
\label{prop:S03:Intertwining}
Let $A$ be a separable \ca{}, and let $A_i$ ($i\in I$) be a collection of separable \ca{s} together with morphisms $\varphi_i\colon A_i\to A$.

Assume that the following holds:
For every $i\in I$, $\varepsilon>0$, $F\ssubset A_i$, $E\ssubset\ker(\varphi_i)$ and $H\ssubset A$, there exist $j\in I$ and a morphism $\psi\colon A_i\to A_j$ such that:
\begin{enumerate}
\item[\approxARef{1}]
\label{defn:S03:A-Approx:commute}
$\varphi_j\circ\psi=_{\varepsilon}^F\varphi_i$,
\item[\approxARef{2}]
\label{defn:S03:A-Approx:kill_kernel}
$\psi=_{\varepsilon}^E0$,
\item[\approxARef{3}]
\label{defn:S03:A-Approx:large_image}
$H\subseteq_\varepsilon\im(\varphi_j)$.
\end{enumerate}

Then $A$ is isomorphic to an inductive limit of some of the algebras $A_i$.
More precisely, there exist indices $i(1),i(2),\ldots\in I$ and morphisms $\psi_k\colon A_{i(k)}\to A_{i(k+1)}$ such that $A\cong\varinjlim_k(A_{i(k)},\psi_k)$.
\end{prp}
\begin{proof}
By induction, we will construct a one-sided approximate intertwining as shown in the following diagram.
This diagram does not commute, but it `approximately commutes'.
\begin{center}
\makebox{
\xymatrix{
A_{i(1)} \ar[r]^{\psi_1} \ar[d]_{\varphi_{i(1)}}
    & A_{i(2)} \ar[r]^{\psi_2} \ar[d]_{\varphi_{i(2)}}
    & A_{i(2)} \ar[r] \ar[d]_{\varphi_{i(3)}}
    & \ldots \ar[r]
    &  B \ar[d]^{\omega}
    \\
A \ar[r]
    & A \ar[r]
    & A \ar[r]
    & \ldots \ar[r]
    &  A
    \\
}}
\end{center}

Property \approxARef{1} is the essential requirement for constructing the one-sided approximate intertwining, that is, to align some of the algebras $A_i$ into an inductive system with limit $B$ together with a canonical morphism $\omega\colon \varinjlim\mathcal{B}\to A$.
Property \approxARef{2} is used to get $\omega$ injective, and \approxARef{2} is used to ensure that $\omega$ is surjective.

More precisely, we proceed as follows:
Let $(x_1,x_2,\ldots)$ be a dense sequence in $A$ with $x_1=0$.
We will construct the following:
\begin{itemize}
\item
indices $i(k)\in I$, for $k\in\NN$ with $k\geq 1$,
\item
morphisms $\psi_k\colon A_{i(k)}\to A_{i(k+1)}$, for $k\geq 1$,
\item
finite subsets $F_k^l\subseteq A_{i(k)}$, for $k,l\geq 1$,
\item
finite subsets $E'_k\subseteq\ker(\varphi_{i(k)})$, for $k\geq 1$,
\end{itemize}
such that
\begin{enumerate}[(a)]
\item
\label{prop:S03:Main_Criterion:cond1}
$\psi_k(F_k^l)\subseteq F_{k+1}^l$, for all $k,l\geq 1$,
\item
\label{prop:S03:Main_Criterion:cond2}
$F_k^1\subseteq F_k^2\subseteq\ldots$, and $\bigcup_{l\geq 1}F_k^l$ is dense in $A_{i(k)}$, for each $k\geq 1$,
\item
\label{prop:S03:Main_Criterion:cond3}
$\dist(a,E_k')\leq\|\varphi_{i(k)}(a)\|+1/2^k$, for all $k\geq 1$ and $a\in F_k^k$,
\item
\label{prop:S03:Main_Criterion:cond4}
$\varphi_{i(k+1)}\circ\psi_k=_{1/2^k}^{F_k^k}\varphi_{i(k)}$, for each $k\geq 1$,
\item
\label{prop:S03:Main_Criterion:cond5}
$\psi_k=_{1/2^k}^{E'_k}0$, for each $k\geq 1$,
\item
\label{prop:S03:Main_Criterion:cond6}
$\{x_1,\ldots,x_k\}\subseteq_{1/2^{k-1}}\varphi_{i(k)}(F_k^k)$, for each $k\geq 1$.
\end{enumerate}

We start by fixing any $i(1)$ in $I$.
Since $x_1=0$, \myRef{prop:S03:Main_Criterion:cond6} is satisfied.
Choose sets $F_1^l$, for $l\geq 1$, satisfying \myRef{prop:S03:Main_Criterion:cond2}.
Then choose $E'_1$ fulfilling property \myRef{prop:S03:Main_Criterion:cond3}.

Let us manufacture the induction step from $k$ to $k+1$.
We consider the index $i(k)$, the tolerance $1/2^{k}$, and the finite sets $F_k^k\subseteq A_{i(k)}$, $E'_k\subseteq\ker(\varphi_{i(k)})$, and $\{x_1,\ldots,x_{k+1}\}\subseteq A$.
By assumption, there is an index $i(k+1)$, and a morphism $\psi_k\colon A_{i(k)}\to A_{i(k+1)}$ satisfying conditions \myRef{prop:S03:Main_Criterion:cond4}, \myRef{prop:S03:Main_Criterion:cond5} and \myRef{prop:S03:Main_Criterion:cond6}.
Choose sets $F_{k+1}^l$, for $l\geq 1$, satisfying properties \myRef{prop:S03:Main_Criterion:cond1} and \myRef{prop:S03:Main_Criterion:cond2}.
Then choose $E_{k+1}'$ fulfilling property \myRef{prop:S03:Main_Criterion:cond3}.

Set $B:=\varinjlim_k(A_{i(k)},\psi_k)$.
Given $k\geq 1$ and $a\in A_{i(k)}$, it follows from \myRef{prop:S03:Main_Criterion:cond1}, \myRef{prop:S03:Main_Criterion:cond2} and \myRef{prop:S03:Main_Criterion:cond4} that $(\varphi_{i(s)}\circ\psi_{s,k})(a)$ is a Cauchy sequence for $s\to\infty$.
We may therefore define morphisms $\omega_k\colon A_{i(k)}\to A$ by
\begin{align*}
\omega_k(a):=\lim_{s\to\infty}(\varphi_{i(s)}\circ\psi_{s,k})(a),
\end{align*}
for $k\geq 1$ and $a\in A_{i(k)}$.
Note that $\omega_l\circ\psi_{l,k}=\omega_k$ for any $k\leq l$.
Thus, the morphisms $\omega_k$ fit together to define a morphism $\omega\colon B\to A$.

Claim.
Let $k\geq 1$.
Then $\|\psi_{\infty,k}(x)\|\leq \|\omega_k(x)\|$, for all $x\in A_{i(k)}$.
Since $\bigcup_l F_k^l$ is dense in $A_{i(k)}$, it is enough to show the inequality for $x\in \bigcup_l F_k^l$.
To prove the claim, let $l\geq 1$ and let $x\in F_k^l$.
Let $n\geq k,l$, and set $b:=\psi_{n,k}(a)$, which belongs to $F_n^n$.
It follows from \myRef{prop:S03:Main_Criterion:cond1} and \myRef{prop:S03:Main_Criterion:cond4} that
\[
\omega_n=^{F_n^n}_{1/2^{n-1}}\varphi_{i(n)}.
\]
We deduce that
\[
\|\varphi_{i(n)}(b)\| < \|\omega_n(b)\| + 1/2^{n-1}.
\]
By \myRef{prop:S03:Main_Criterion:cond3}, we obtain $e\in E_n'$ satisfying
\[
\|b-e\|\leq\|\varphi_{i(n)}(b)\|+1/2^n.
\]
By \myRef{prop:S03:Main_Criterion:cond5}, we have $\|\psi_n(e)\|<1/2^n$.
We deduce that
\begin{align*}
\|\psi_{\infty,k}(a)\|
&\leq \|\psi_n(b)\|
= \|\psi_n(b-e+e)\|
\leq \|b-e\| + \|\psi_n(e)\| \\
&\leq \|\varphi_{i(n)}(b)\| + 1/2^n+1/2^n
\leq \|\omega_n(b)\| + 1/2^{n-2}
= \|\omega_k(a)\| + 1/2^{n-2}.
\end{align*}
Since this holds for all $n\geq k,l$, we obtain the claimed inequality.

It follows from the claim that $\omega$ is injective.
To verify that $\omega$ is surjective, let $a\in A$ and let $\varepsilon>0$.
Since the sequence $x_1,x_2,\ldots$ is dense in $A$, there exists some $l\geq 1$ with $\|a-x_l\|<\varepsilon/3$.
Let $k\geq l$ be a number with $1/2^{k-1}<\varepsilon/3$.
As noted in the proof of the above claim, we have $\omega_k=_{1/2^{k-1}}^{F_k^k}\varphi_{i(k)}$.
Using this at the third step, and using \myRef{prop:S03:Main_Criterion:cond6} at the second step, we deduce that
\begin{align*}
a
\;\;\in_{\varepsilon/3}\;\;\{x_1,\ldots,x_k\}
\;\;\subseteq_{1/2^{k-1}}\;\;\varphi_{i(k)}(F_k^k)
\;\;\subseteq_{1/2^{k-1}}\;\;\omega_k(F_k^k)
\;\;\subseteq\;\;\im(\omega).
\end{align*}
Together, $a$ lies in $\im(\omega)$ up to $\varepsilon/4+1/2^{k-1}+1/2^{k-1}<\varepsilon$.
Since $\varepsilon>0$ was arbitrary, we deduce that $a\in\im(\omega)$.
\end{proof}

\begin{pgr}
Let us consider a weaker approximation than in \autoref{prop:S03:Intertwining}, where we relax condition \approxARef{3}.
Let us assume that the following situation is given:

Let $A$ be a separable \ca{}, and let $A_i$ ($i\in I$) be a collection of separable \ca{s} together with morphisms $\varphi_i\colon A_i\to A$, such that the following holds:
For every $i\in I$, $\varepsilon>0$, $F\ssubset A_i$ and $E\ssubset\ker(\varphi_i)$, there exist $j\in J$ and a morphism $\psi\colon A_i\to A_j$ such that:
\begin{enumerate}
\item[\approxARef{1}]
\label{defn:S03:B-Approx:commute}
$\varphi_j\circ\psi=_{\varepsilon}^F\varphi_i$,
\item[\approxARef{2}]
\label{defn:S03:B-Approx:kill_kernel}
$\psi=_{\varepsilon}^E0$,
\end{enumerate}
and moreover, the following condition holds:
\begin{enumerate}
\item[\approxARef{3'}]
\label{defn:S03:B-Approx:large_image}
the collection $\{\im(\varphi_i):i\in I\}$ approximates $A$.
\end{enumerate}

Condition \approxARef{3} of \autoref{prop:S03:Intertwining} is a statement about the morphism $\psi$.
It roughly says that $\psi$ has `large' image.
The above condition \approxARef{3'} is independent of the morphism $\psi$.
It just requires that the collection of all sub-\ca{s} $\im(\varphi_i)$ is `large'.

Adopting the proof of \autoref{prop:S03:Intertwining}, we may construct one-sided approximate intertwinings to get the following result:
For every $\gamma>0$ and $H\ssubset A$, there exists a sub-\ca{} $B\subseteq A$ such that $H\subseteq_\gamma B$ and such that $B$ is an inductive limit of some of the algebras $A_i$.

If we denote by $\mathcal{C}=\{A_i : i\in I\}$ the class of approximating algebras, then this means precisely that $A$ is $A\mathcal{C}$-like, that is, $A$ is approximated by sub-\ca{s} that are inductive limits of algebras in $\mathcal{C}$.
In general, this does not imply that $A$ is an $A\mathcal{C}$-algebra.
In fact, not even a $\mathcal{C}$-like \ca{} need to be an $A\mathcal{C}$-algebra, as can be seen by the following example.
\end{pgr}

\begin{exa}[{Dadarlat, Eilers, \cite{DadEil99AHNotLocal}}]
\label{exmpl:S03:Counterexmpl_limit_of_limit}
Let us denote by $H$ the class of (direct sums of) homogeneous \ca{s}.
An inductive limit of \ca{s} in $H$ is called an $AH$-algebra.
In \cite{DadEil99AHNotLocal}, Dadarlat and Eilers construct a \ca{} $A=\varinjlim_k A_k$ that is an inductive limit of $AH$-algebras $A_k$ (so $A$ is an $AAH$-algebra) but such that $A$ is not an $AH$-algebra itself.
Thus, an $AA\mathcal{C}$-algebra need not be an $A\mathcal{C}$-algebra in general.

Since quotients of homogeneous algebras are again homogeneous, the \ca{} $A$ is also $H$-like.
Thus, the example also shows that a $\mathcal{C}$-like algebra need not be an $A\mathcal{C}$-algebra in general.

In the example of Dadarlat and Eilers, each \ca{} $A_k$ is an inductive limit, $\varinjlim_n A_k^n$, of \ca{s} $A_k^n$ that have the form $\bigoplus_{i=1}^d M_{d_i}(C(X_i))$ with each $X_i$ a three-dimensional CW-complex.
It is well-known that $C(X_i)$ is not weakly semiprojective if $X_i$ contains a copy of the two-dimensional disc;
see for example \cite[Remark~3.3]{SoeThi12CharCommutSP}.
It follows that the algebras $A_k^n$ are not weakly semiprojective.
This is the crucial point, as will be shown in \autoref{prop:S03:AC-like_implies_AC} and \autoref{prop:S03:AAC_implies_AC}.
\end{exa}

The following \autoref{prop:S03:twist_maps_into_subalg} is a variant of \cite[Lemma~15.2.1, p.118]{Lor97LiftingSolutions} that avoids the assumption of finite generation, see \autoref{pargr:S02:Generators}.
It was shown to the author by Chigogidze and Loring (private communication).
The result is used in the proof of \autoref{prop:S03:AC-like_implies_AC} to `twist' morphisms from weakly semiprojective \ca{s}.
We include a proof for completeness.

\begin{lma}[{Chigogidze and Loring}]
\label{prop:S03:twist_maps_into_subalg}
Let $A$ be a separable, weakly semiprojective \ca{}.
Then for every $\varepsilon>0$ and $F\ssubset A$, there exist $\delta>0$ and $G\ssubset A$ such that the following holds:
Whenever $\varphi\colon A\to B$ is a morphism to another \ca{} $B$, and whenever $C\subseteq B$ is a sub-\ca{} with $\varphi(G)\subseteq_\delta C$, then there exists a morphism $\psi\colon A\to C$ such that $\psi=_\varepsilon^F\varphi$.
\end{lma}
\begin{proof}
Let $A$, $F$ and $\varepsilon$ be as in the statement.
Let $P:=C^\ast(x_1,x_2,\ldots : \|x_i\|\leq 1)$ be the universal \ca{} generated by a sequence of contractions.
Since $A$ is separable, there exists a surjective morphism $\pi\colon P\to A$.
Denoting the cardinality of $F$ by $|F|$, we let $F':=\{x_1,\ldots,x_{|F|}\}\subseteq P$.
We may choose $\pi$ such that $F=\pi(F')$.
We denote the kernel of $\pi$ by $J$.
Since $J$ is separable, there exists a strictly positive element $h$ in $J$.
Set $J_0=\{0\}$ and for each $k\geq 1$ let $J_k$ be the ideal of $P$ generated by $(h-\tfrac{1}{k})_+$.
Then $(J_k)_k$ is an increasing sequence of ideals with $J=\overline{\bigcup_k J_k}$.
We consider this as an inductive system and we let $\pi_{l,k}\colon P/J_k\to P/J_l$ denote the (surjective) connecting morphisms.
With this notation, $\pi=\pi_{\infty,0}$.

Since $A$ is weakly semiprojective, there exist $n'$ and a morphism $\sigma'\colon A\to P/J_{n'}$ such that $\pi_{\infty,n'}\circ\sigma'=_{\varepsilon/2}^F\id_A$.
Then $\pi_{\infty,n'}\circ\sigma'\circ\pi=_{\varepsilon/2}^{F'}\pi$.
For each $i=1,\ldots,|F|$, we have
\[
\varepsilon/2
> \left\| \big( \pi_{\infty,n'}\circ\sigma'\circ\pi\big)(x_i)-\pi(x_i) \right\|
= \lim_{n\geq n'} \left\| \big(\pi_{n,n'}\circ\sigma'\circ\pi\big)(x_i)-\pi_{n,0}(x_i) \right\|.
\]
Thus, there is $n\geq n'$ such that $\pi_{n,n'}\circ\sigma'\circ\pi=_{\varepsilon/2}^{F'}\pi_{n,0}$.
Set $\sigma:=\pi_{n,n'}\circ\sigma'\colon A\to P/J_n$.
Then $\sigma\circ\pi=_{\varepsilon/2}^{F'}\pi_{n,0}$.
The situation and the maps to be constructed are shown in the following diagram.
\begin{center}
\makebox{
\xymatrix{
& P \ar[dr]^{\gamma} \ar[d]_{\pi_{n,0}} \\
& P/J_n \ar[r]^{\tilde{\gamma}} \ar[d]^{\pi_{\infty,n}} & C \ar@{^{(}->}[d] \\
A \ar@{}[r]|{\cong} & P/J \ar[r]^{\varphi} \ar@/^1pc/[u]^{\sigma} & B
}}
\end{center}

Since the elements $x_i$ generate $P$, there exist $d\in\NN$ and a ${}^*$-polynomial $p$ in $d$ non-commuting variables such that $h=_{1/3n}p(x_1,\ldots,x_d)$.
Then there exists $\delta>0$ with the following property:
If $D$ is a \ca{}, and if $y_i\in D$ and $z_i\in D$ are contractive elements with $y_i=_\delta z_i$ for $i=1,\ldots,d$, then $p(y_1,\ldots,y_d)=_{1/3n}p(z_1,\ldots,z_d)$.
(Note that $\delta$ depends only on $p$ and not on the \ca{} $D$ or the elements.)
We may assume that $F'\subseteq\{x_1,\ldots,x_d\}$ and $\delta<\varepsilon/2$.
Set $G=\{\pi(x_1),\ldots,\pi(x_d)\}$.

We claim that $G$ and $\delta$ have the desired properties.
So let $\varphi\colon A\to B$ be a morphism, and let $C\subseteq B$ be a sub-\ca{} with $\varphi(G)\subseteq_\delta C$.
This means that there exist elements $c_i\in C$ with $c_i=_\delta(\varphi\circ\pi)(x_i)$ for $i=1,\ldots,d$.
We define a morphism $\gamma\colon P\to C$ by sending $x_i$ to $c_i$ for $i\leq d$ and to $0\in C$ for $i>l$.
Note that this implies $\gamma=_{\varepsilon/2}^{F'}\varphi\circ\pi$.

By choice of $\delta$, we get
\[
p(c_1,\ldots,c_d)
=_{1/3n} p\big( (\varphi\circ\pi)(x_1),\ldots,(\varphi\circ\pi)(x_d) \big)
= (\varphi\circ\pi)(p(x_1,\ldots,x_d)).
\]
Using this at the third step, we obtain:
\begin{align*}
\gamma(h)
&=_{1/3n} \gamma(p(x_1,\ldots,x_d))
= p(c_1,\ldots,c_d)
=_{1/3n} (\varphi\circ\pi)(p(x_1,\ldots,x_d)) \\
&=_{1/3n} (\varphi\circ\pi)(h).
\end{align*}
Since $\pi(h)=0$, we get $\|\gamma(h)\|<1/n$.
It follows that $\gamma((h-\tfrac{1}{n})_+)=0$, and so the kernel of $\gamma$ contains $J_n$.
Thus, there exists a morphism $\tilde{\gamma}\colon P/J_n\to C$ such that $\gamma=\tilde{\gamma}\circ\pi_{n,0}$.

Set $\psi:=\tilde{\gamma}\circ\sigma\colon A\to C$.
Using that $\sigma\circ\pi=_{\varepsilon/2}^{F'}\pi_{n,0}$ at the second step, we get
\[
\psi\circ\pi
=\tilde{\gamma}\circ\sigma\circ\pi
=_{\varepsilon/2}^{F'}\tilde{\gamma}\circ\pi_{n,0}
=\gamma
=_{\varepsilon/2}^{F'}\varphi\circ\pi.
\]
It follows that $\psi=_\varepsilon^{F}\varphi$, as desired.
\end{proof}

\begin{thm}
\label{prop:S03:AC-like_implies_AC}
Let $\mathcal{C}$ be a family of separable, weakly semiprojective \ca{s}.
Then every separable $A\mathcal{C}$-like \ca{} is already an $A\mathcal{C}$-algebra.
\end{thm}
\begin{proof}
Let $A$ be an $A\mathcal{C}$-like \ca{}.
We want to apply the one-sided approximate intertwining, \autoref{prop:S03:Intertwining}, to show that $A$ is an $A\mathcal{C}$-algebra.
For this we consider the collection of all morphisms $\varphi\colon C\to A$ where $C$ is a \ca{} from $\mathcal{C}$ (we may think of this collection as being indexed over $\coprod_{C\in\mathcal{C}}\text{Hom}(C,A)$).

We need to check the requirements for \autoref{prop:S03:Intertwining}.
So assume the following data is given:
A morphism $\varphi\colon C\to A$ with $C\in\mathcal{C}$, $\varepsilon>0$, $F\ssubset C$, $E\ssubset\ker(\varphi)$ and $H\ssubset A$.
We may assume that $F$ contains $E$.
We need to find a \ca{} $C'\in\mathcal{C}$ together with a morphism $\varphi'\colon C'\to A$, and a morphism $\psi\colon C\to C'$ such that \approxARef{1}, \approxARef{2} and \approxARef{3} are satisfied.

Applying \autoref{prop:S03:twist_maps_into_subalg} to the weakly semiprojective \ca{} $C$ for $\varepsilon/3$ and $F\subseteq C$, we obtain $\delta>0$ and $G\ssubset C$ such that any morphism out of $C$ that maps $G$ up to $\delta$ into a given sub-\ca{} can be twisted to map exactly into that sub-\ca{} while moving $F$ at most by $\varepsilon$.
We may assume that $\delta\leq\varepsilon/3$.

Set $H':=H\cup\varphi(G)$, which is a finite subset of $A$.
By assumption, there exists a sub-\ca{} $B\subseteq A$ that contains $H'$ up to $\delta$ and which is an $A\mathcal{C}$-algebra, say $B=\varinjlim_k C_k$ with connecting morphisms $\gamma_k\colon C_k\to C_{k+1}$.
Since $\varphi(G)\subseteq_\delta B$, there exists a morphism $\alpha\colon C\to B$ such that $\varphi=_{\varepsilon/3}^F\alpha$.
The situation and the maps to be constructed are shown in the following diagram.
\begin{center}
\makebox{
\xymatrix{
C_{k_1} \ar[r]^{\gamma_{k_2,k_1}}
& C_{k_2} \ar[r]^-{\gamma_{k_3,k_2}}
& C_{k_3}=C' \ar[r]^-{\gamma_{\infty,k_3}}
& B \ar@{^{(}->}[r]
& A \\
& & &
C \ar@{..>}[ulll]^{\widetilde{\alpha}} \ar@{..>}[ul]^{\psi} \ar@{..>}[u]^{\alpha} \ar[ur]_{\varphi}
}}
\end{center}

By \autoref{prop:S02:WSP_lifting_limits}, the morphism $\alpha\colon C\to B=\varinjlim_k C_k$ has an approximate lift, that is, there exist an index $k_1$ and a morphism  $\widetilde{\alpha}\colon C\to C_{k_1}$ such that $\alpha=_{\varepsilon/3}^F\gamma_{\infty,k_1}\circ\widetilde{\alpha}$.
Then $\varphi=_{2\varepsilon/3}^F\gamma_{\infty,k_1}\circ\widetilde{\alpha}$.
Upon going further down in the inductive limit, we can guarantee the properties that we need to check.

In order to guarantee \approxARef{2}, we consider $E\ssubset\ker(\varphi)$.
Since $\varphi=_{2\varepsilon/3}^F\gamma_{\infty,k_1}\circ\widetilde{\alpha}$ and $E\subseteq F$, we have
\[
\gamma_{\infty,k_1}\circ\widetilde{\alpha}=_{2\varepsilon/3}^E\varphi=^E0.
\]
Thus, we may choose $k_2\geq k_1$ such that $\gamma_{k_2,k_1}\circ\widetilde{\alpha}=_{\varepsilon}^E0$.

In order to guarantee \approxARef{3}, we consider $H\ssubset A$.
Since $H\subseteq_\delta B=\varinjlim_k C_k$, we may find $k_3\geq k_2$ such that $H\subseteq_{2\delta}\im(\gamma_{\infty,k_3})$.

Set $C':=C_{k_3}$, $\varphi':=\gamma_{\infty,k_3}$ and $\psi:=\gamma_{k_3,k_1}\circ\widetilde{\alpha}\colon C\to C'=C_{k_3}$.
By construction, \approxARef{1}, \approxARef{2} and \approxARef{3} are satisfied.
\end{proof}

We obtain the following generalization of \cite[Lemma~15.2.2, p.119]{Lor97LiftingSolutions}.

\begin{cor}
\label{prop:S03:C_like_implies_AC}
Let $\mathcal{C}$ be a class of separable, weakly semiprojective \ca{s}.
Then every separable $\mathcal{C}$-like \ca{} is an $A\mathcal{C}$-algebra.
\end{cor}

\begin{rmk}
\label{pargr:S03:ACC_implies_AC-like}
Let $\mathcal{C}$ be a class of separable, weakly semiprojective \ca{s}.
If $\mathcal{C}$ is closed under quotients, then every $A\mathcal{C}$-like \ca{} is also $\mathcal{C}$-like, and similarly every $AA\mathcal{C}$-algebra is $\mathcal{C}$-like.
Then Theorems~\ref{prop:S03:AC-like_implies_AC} and~\ref{prop:S03:AAC_implies_AC} follow from \autoref{prop:S03:C_like_implies_AC}.

In \autoref{sect:S04:trivial_shape} we will consider the class $\mathcal{P}$ of projective \ca{s}.
This class is not closed under quotients.
Indeed, there exist $A\mathcal{P}$-like \ca{s} that are not $\mathcal{P}$-like:
Consider for example the commutative \ca{} $A=C_0([0,1]^2\setminus\{(0,0)\})$, which is contractible and hence $A\mathcal{P}$-like (even an $A\mathcal{P}$-algebra) by \autoref{prop:S04:contr_AP}.
Every sub-\ca{} of $A$ is commutative, and it was shown by Chigogidze and Dranishnikov, \cite{ChiDra10NcAR}, that every commutative projective \ca{} has one\--di\-men\-sional spectrum.
In particular, every commutative projective \ca{} has stable rank one, and if $A$ was approximated by such sub-\ca{s}, then $A$ would have stable rank one as well, which contradicts the fact that the stable rank of $A$ is two.

Therefore, in order to obtain \autoref{prop:S04:Permanence_trivSh} $(2)$ and $(3)$, it is crucial that Theorems~\ref{prop:S03:AC-like_implies_AC} and~\ref{prop:S03:AAC_implies_AC} hold for classes that are not necessarily closed under quotients.
\end{rmk}

\begin{thm}
\label{prop:S03:AAC_implies_AC}
Let $\mathcal{C}$ be a class of separable, weakly semiprojective \ca{s}.
Then every separable $AA\mathcal{C}$-algebra is already an $A\mathcal{C}$-algebra.
\end{thm}
\begin{proof}
Let $A$ be an $AA\mathcal{C}$-algebra.
Choose an inductive system $(A_k,\gamma_k)$ of algebras in $\mathcal{C}$ such that $A\cong\varinjlim_k A_k$.
For each $k$, choose an inductive system $(A_k^{(n)},\varrho_k^{(n)})$ of algebras in $\mathcal{C}$ such that $A_k\cong\varinjlim_n A_k^{(n)}$.
We are given the following situation:
\begin{center}
\makebox{
\xymatrix{
A_k^{(n)} \ar[d]^{\varrho_k^{(n)}}
    & A_{k+1}^{(n)} \ar[d]^{\varrho_{k+1}^{(n)}}
    & A_{k+2}^{(n)} \ar[d]^{\varrho_{k+2}^{(n)}} \\
A_k^{(n+1)} \ar[d]^{\varrho_k^{(\infty,n+1)}}
    & A_{k+1}^{(n+1)} \ar[d]^{\varrho_{k+1}^{(\infty,n+1)}}
    & A_{k+2}^{(n+1)} \ar[d]^{\varrho_{k+2}^{(\infty,n+1)}}  \\
A_k \ar[r]_{\gamma_{k}}
    & A_{k+1} \ar[r]_{\gamma_{k+1}}
    & A_{k+2} \ar[r] \ar[r]_{\gamma_{\infty,k+2}}
    & A \\
}}
\end{center}
We want to apply \autoref{prop:S03:Intertwining} for the collection of \ca{s} $A_k^{(n)}$ together with morphisms $\varphi_{k,n}:=\gamma_{\infty,k}\circ\varrho_k^{(\infty,n)}\colon A_k^{(n)}\to A$, for $k,n\geq 1$.

Assume some indices $k, n$ are given together with $\varepsilon>0$, $F\ssubset A_k^{(n)}$, $E\ssubset\ker(\varphi_{k,n})$ and $H\ssubset A$.
We may assume that $E\subseteq F$.
We need to find $k',n'\geq 1$ and a morphism $\psi\colon A_k^{(n)}\to A_{k'}^{(n')}$ that satisfy \approxARef{1}, \approxARef{2} and \approxARef{3}.

Since $A=\varinjlim_k A_k$ and $\varphi_{k,n}=\gamma_{\infty,k}\circ\varrho^{(\infty,n)}_k=^E0$, there exists some $k'\geq k$ such that $\gamma_{k',k}\circ\varrho^{(\infty,n)}_k=^E_{\varepsilon/2} 0$.
We can also ensure that $H\subseteq_{\varepsilon}\im(\gamma_{\infty,k'})$, by further increasing $k'$, if necessary.

Since $A_k^{(n)}$ is weakly semiprojective, we may lift the morphism
\[
\gamma_{k',k}\circ\varrho_k^{(\infty,n)}\colon A_k^{(n)}\to A_{k'}=\varinjlim_n A_{k'}^{(n)}
\]
to $\alpha\colon A_k^{(n)}\to A_{k'}^{(n_1)}$ (for some $n_1$) such that $\varrho_{k'}^{(\infty,n_1)}\circ\alpha=^F_{\varepsilon/2}\gamma_{k',k}\circ\varrho_k^{(\infty,n)}$.
This is shown in the following diagram.
\begin{center}
\makebox{
\xymatrix{
& A_{k'}^{(n_1)} \ar[d]^{\varrho_{k'}^{(n',n_1)}} \\
A_k^{(n)} \ar[d]_{\varrho_k^{(\infty,n)}} \ar@{..>}[ur]^{\alpha} \ar[r]^{\psi}
& A_{k'}^{(n')} \ar[d]^{\varrho_{k'}^{(\infty,n')}} \\
A_k \ar[r]_{\gamma_{k',k}}
& A_{k'} \ar[r]_{\gamma_{\infty,k'}}
& A \\
}}
\end{center}

We have $\varrho_{k'}^{(\infty,n_1)}\circ\alpha =^E_{\varepsilon/2}\gamma_{k',k}\circ\varrho_k^{(\infty,n)} =^E_{\varepsilon/2} 0$.
As in the proof of \autoref{prop:S03:AC-like_implies_AC}, by going further down in the inductive limit we may find $n'\geq n_1$ such that
\[
\varrho_{k'}^{(n',n_1)}\circ\alpha=^E_{\varepsilon} 0, \quad \text{ and } \quad H\subseteq_\varepsilon\im(\gamma_{\infty,k'}\circ\varrho_{k'}^{(\infty,n')}).
\]
Set $\psi:=\varrho_{k'}^{(n',n_1)}\circ\alpha\colon A_k^{(n)}\to A_{k'}^{(n')}$.
By construction, \approxARef{1}, \approxARef{2} and \approxARef{3} are satisfied.
\end{proof}

\begin{pgr}
Let $B$ be a separable \ca{}, and $\mathcal{C}$ a class of separable \ca{s}.
The above results give us connections between the four conditions that $B$ is $\mathcal{C}$-like, or $A\mathcal{C}$-like, or an $A\mathcal{C}$-algebra, or an $AA\mathcal{C}$-algebra.
This is shown in the diagram below.
A dotted arrow indicates that the implication holds under the additional assumption that the algebras in $\mathcal{C}$ are weakly semiprojective.
The dashed arrow with $(\ast)$ indicates that the implication holds if each quotient of an algebra in $\mathcal{C}$ is an $A\mathcal{C}$-algebra, while the dashed arrow with $(\ast\ast)$ indicates that the implication holds if $\mathcal{C}$ is closed under quotients;
see also \autoref{pargr:S03:ACC_implies_AC-like}.

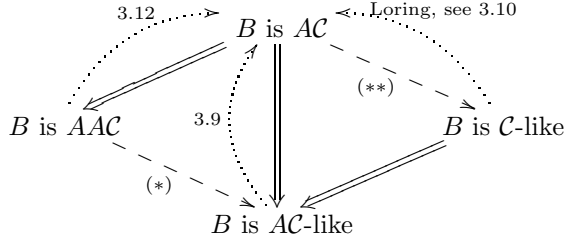
\begin{center}
\makebox{
\xymatrix{
    & \txt{ $B$ is $A\mathcal{C}$ } \ar@{=>}[dl] \ar@{=>}[dd]
    \ar@{-->}[dr]_{(\ast\ast)} \\
\txt{ $B$ is $AA\mathcal{C}$ } \ar@{-->}[dr]_{(\ast)}
    \ar@/^1.5pc/@{..>}[ur]^{\text{\ref{prop:S03:AAC_implies_AC}}}
    & &  \txt{ $B$ is $\mathcal{C}$-like } \ar@{=>}[dl] \ar@/_1.5pc/@{..>}[ul]_{\text{Loring, see \ref{prop:S03:C_like_implies_AC}}} \\
    & \txt{ $B$ is $A\mathcal{C}$-like } \ar@/^1.5pc/@{..>}[uu]^{\text{\ref{prop:S03:AC-like_implies_AC}}} \\
    \\
}}
\end{center}
\end{pgr}

\section{Trivial shape}
\label{sect:S04:trivial_shape}

In this section we study \ca{s} that are shape equivalent to the zero \ca{}.
Such algebras are said to have \emph{trivial shape}.
In \autoref{prop:S04:TFAE-trivSh}, we show that having trivial shape is equivalent to several other natural conditions, most importantly to being an inductive limit of projective \ca{s}.
One may further obtain that the connecting morphisms in such an inductive limit are surjective, see \autoref{prop:S04:change_connecting_mor_surj}.

In \autoref{prop:S04:Permanence_trivSh}, we prove some natural permanence properties of trivial shape.
However, building on an example of Dadarlat, \cite{Dad14StblyContrNotContr}, we show that trivial shape does not necessarily pass to full hereditary sub-\ca{s};
see \autoref{pargr:S04:exmpl_Dadarlat}.
It follows that also projectivity does not pass to full hereditary sub-\ca{s};
see \autoref{prop:S04:Projectivity_not_hereditary}.

We denote the zero \ca{} by $0$.
Note that $A\precsim_{Sh}0$ implies $A\sim_{Sh}0$, that is, $A$ is shape dominated by $0$ if and only if it is shape equivalent to $0$.
The following result of Loring and Shulman was the inspiration for the main result \autoref{prop:S04:TFAE-trivSh} below.
For the definition of the generating rank, $\gen(A)$, see \autoref{pargr:S02:Generators}.
The \emph{cone} of a \ca{} $A$, denoted $\cone{A}$, is defined as $\cone{A}:=C_0((9,1])\otimes A$.

\begin{thm}[{Loring, Shulman, \cite[Theorem~7.4]{LorShu12NCSemialgLifting}}]
\label{prop:S04:LorShu:Cone_as_limit}
Let $A$ be a separable \ca{}.
Then $\cone{A}\cong \varinjlim_k P_k$, for an inductive system of projective \ca{s} $P_k$ satisfying $\gen(P_k)\leq\gen(A)+1$, with surjective connecting morphisms $P_k\to P_{k+1}$.
\end{thm}

\begin{lma}
\label{prop:S04:lma:proj_map_null-homotopic}
Let $\varphi\colon A\to B$ be a projective morphism.
Then $\varphi\simeq 0$.
\end{lma}
\begin{proof}
This is a variant of the standard argument for showing that a projective \ca{} is contractible.
We include it for completeness.
For the cone $\cone{B}=C_0((0,1],B)$, let $\ev_1\colon\cone{B}\to B$ be the evaluation morphism at $1$.
The projectivity of $\varphi$ gives a lift $\psi\colon A\to \cone{B}$ such that $\ev_1\circ\psi=\varphi$.
We have $\id_{\cone{B}}\simeq 0$ since $\cone{B}$ is contractible.
Then $\varphi=\ev_1\circ\id_{\cone{B}}\circ\psi\simeq 0$, as desired.
\end{proof}

\begin{lma}
\label{prop:S04:lma:shape_subsystem_null-homotopic}
Let $(A_k,\gamma_k)$ be a shape system with inductive limit $A:=\varinjlim A_k$.
Assume that every semiprojective morphism $D\to A$ (from any \ca{} $D$) is null-homotopic.
Then for each $k$ there exists $k'\geq k$ such that $\gamma_{k',k}\simeq 0$.
\end{lma}
\begin{proof}
The morphisms to be considered are shown in the following diagram:
\begin{center}
\makebox{
\xymatrix@+=5pt{
A_{k} \ar[r]^{\gamma_{k+1,k}}
    & A_{k+1} \ar@<2pt>[d]^{\sigma_2} \ar@<-2pt>[d]_{\sigma_1}
    \\
& A_{k+2} \ar[r]_{\gamma_{k',k+2}}
    & A_{k'} \ar[r]_{\gamma_{\infty,k'}}
    & A
}}
\end{center}

Note that $\gamma_{k+1,k}$ is semiprojective.
Define two morphisms $\sigma_1,\sigma_2\colon A_{k+1}\to A_{k+2}$ as $\sigma_1=\gamma_{k+2,k+1}$ and $\sigma_2=0$.
The morphism $\gamma_{\infty,k+2}\circ\sigma_1=\gamma_{\infty,k+1}$ is semiprojective, and therefore null-homotopic by assumption.
Thus $\gamma_{\infty,k+2}\circ\sigma_1\simeq 0=\gamma_{\infty,k+2}\circ\sigma_2$.

Using the semiprojectivity of $\gamma_{k+1,k}$, we may deduce from \cite[3.2]{EffKam86ShapeThy}, see \autoref{prop:S02:SP_lifting_limits}, that there exists $k'\geq k+2$ such that $\gamma_{k',k}=\gamma_{k',k+2}\circ\sigma_1\circ\gamma_{k+1,k} \simeq\gamma_{k',k+2}\circ\sigma_2\circ\gamma_{k+1,k}=0$.
\end{proof}

\begin{thm}
\label{prop:S04:TFAE-trivSh}
Let $A$ be a separable \ca{}.
Then the following are equivalent:
\begin{enumerate}
\item \label{prop:S04:TFAE-trivSh:conda}
$A$ has trivial shape, that is, $A\sim_{Sh}0$.
\item \label{prop:S04:TFAE-trivSh:condb}
Every semiprojective morphism $D\to A$ (from any \ca{} $D$) is null-homotopic.
\item \label{prop:S04:TFAE-trivSh:condc}
$A$ is an inductive limit of a system with projective connecting morphisms.
\item \label{prop:S04:TFAE-trivSh:condd}
$A$ is an inductive limit of a system with null-homotopic connecting morphisms.
\item \label{prop:S04:TFAE-trivSh:conde}
$A$ is an inductive limit of finitely generated, projective \ca{s} with generating rank at most $\gen(A)+1$.
\item \label{prop:S04:TFAE-trivSh:condf}
$A$ is an inductive limit of finitely generated cones.
\item \label{prop:S04:TFAE-trivSh:condg}
$A$ is an inductive limit of contractible \ca{s}.
\end{enumerate}
\end{thm}
\begin{proof}
Note that $0$ has a natural shape system consisting of the zero \ca{} at each step.
Therefore, $A\sim_{Sh}0$ means that there exists a shape system $(A_k,\gamma_k)$ for $A$ and morphisms $\alpha_k\colon A_k\to 0$ and $\beta_k\colon 0\to A_{k+1}$ such that $\beta_{k+1}\circ\alpha_k\simeq\gamma_k$.
This is shown in the following diagram, which homotopy commutes:
\begin{center} \makebox{
\xymatrix{
A_1 \ar[rr]^{\gamma_1} \ar[dr]_{\alpha_1}
& & A_2 \ar[rr]^{\gamma_2} \ar[dr]_{\alpha_2}
& & A_3 \ar[rr] \ar[dr]_{\alpha_3}
& & \ldots \ar[r]
& A
\\
& 0 \ar[rr] \ar[ur]^{\beta_1}
& & 0 \ar[rr] \ar[ur]^{\beta_2}
& & 0 \ar[r]
& 0 \ar[r]
& 0
\\
} }
\end{center}

`(1)$\Rightarrow$(4)':
Assume that $A\sim_{Sh}0$.
We have just noted that this implies that $A$ has a shape system $(A_k)_k$ with connecting morphisms $\gamma_k\colon A_k\to A_{k+1}$ that are null-homotopic since they factor through $0$ up to homotopy.

`(4)$\Rightarrow$(1)':
Assume that there is an inductive system $\mathcal{A}=(A_k,\gamma_k)$ with $A\cong\varinjlim\mathcal{A}$ and null-homotopic connecting morphisms $\gamma_k$.
Let $\alpha_k\colon A_k\to 0$ and $\beta_k\colon 0\to A_{k+1}$ be the zero morphisms.
Then $\beta_{k+1}\circ\alpha_k=0\simeq\gamma_k$.
Conversely also $\beta_k\circ\alpha_k=0$, so that the inductive systems $\mathcal{A}$ and $(0\to 0\to\ldots)$ are shape equivalent.
This does not show that $A\sim_{Sh}0$ right away since the inductive system $\mathcal{A}$ need not be a shape system.
However, by \cite[Theorem~4.8]{Bla85ShapeThy}, whenever two inductive systems are shape equivalent, then their inductive limit \ca{s} are shape equivalent.

`(4)$\Rightarrow$(2)':
Assume that $A\cong\varinjlim A_k$, for an inductive system with null-homotopic connecting morphisms $\gamma_k\colon A_k\to A_{k+1}$.
Let $\varphi\colon D\to A$ be a semiprojective morphism.
By \cite[Theorem~3.1]{Bla85ShapeThy}, see \autoref{prop:S02:SP_lifting_limits}, there exists $k$ and a morphism $\psi\colon D\to A_k$ such that $\varphi\simeq\gamma_{\infty,k}\circ\psi$.
Since $\gamma_{\infty,k}=\gamma_{\infty,k+1}\circ\gamma_{k}$ and $\gamma_k\simeq 0$, we obtain that $\varphi\simeq 0$.

`(2)$\Rightarrow$(5),(6)':
By \cite[Theorem 4.3]{Bla85ShapeThy}, see \autoref{pargr:S02:shape_system}, $A$ has a shape system $(A_k,\gamma_k)$ with finitely generated algebras $A_k$ and such that $\gen(A_k)\leq\gen(A)$.
We may apply \autoref{prop:S04:lma:shape_subsystem_null-homotopic} inductively to this shape system, and after passing to a suitable subsystem we see that there exists a shape system $(A_k,\gamma_k)$ of finitely generated \ca{s} $A_k$ with $\gen(A_k)\leq\gen(A)$ and null-homotopic connecting morphisms $\gamma_k$ such that $A\cong\varinjlim\mathcal{A}$.

A homotopy $\gamma_k\simeq 0$ induces a natural morphism $\Gamma_k\colon A_k\to \cone{A_{k+1}}$ such that $\gamma_k$ has a factorization $\gamma_k=\ev_1\circ\Gamma_k$, where $\cone{A_{k+1}}$ is the cone over $A_{k+1}$ and $\ev_1$ is evaluation at $1$.
Set $\omega_k:=\Gamma_k\circ\ev_1\colon \cone{A_k}\to \cone{A_{k+1}}$.
Consider the inductive system $\mathcal{B}=(\cone{A_k},\omega_k)$.
The systems $\mathcal{A}$ and $\mathcal{B}$ are intertwined, which implies that their inductive limits are isomorphic.
It follows from \cite[Lemma~7.1]{LorShu12NCSemialgLifting} that $\gen(\cone{A_k})\leq\gen(A_k)+1$.
Thus, $A$ is isomorphic to an inductive limit of the finitely generated cones $\cone{A_k}$, which verifies~(6).
The intertwining is shown in the following commutative diagram.
\begin{center}
\makebox{
\xymatrix{
& A_1 \ar[dr]_{\Gamma_1} \ar[rr]^{\gamma_1}
    & & A_{2} \ar[dr]_{\Gamma_2} \ar[rr]
    & & \ldots \ar[r]
    & \varinjlim_k A_k \ar@<2pt>[d]^{\cong} \\
CA_1 \ar[ur]^{\ev_{1}} \ar[rr]_{\omega_1}
    & & CA_2 \ar[ur]^{\ev_{1}} \ar[rr]_{\omega_2}
    & & CA_3 \ar[r]
    & \ldots \ar[r]
    & \varinjlim_k \cone{A_k} \ar@<2pt>[u]^{\cong} \\
}}
\end{center}

By the result of Loring and Shulmann, \cite[Theorem~7.4]{LorShu12NCSemialgLifting}, see \autoref{prop:S04:LorShu:Cone_as_limit}, for each $k$, the cone $\cone{A_k}$ can be written as an inductive limit of finitely generated projective \ca{s} with generating rank at most $\gen(A_k)+1$.
Note that $\gen(A_k)+1\leq\gen(A)+1$ for all $k$.
It follows from \autoref{prop:S03:AAC_implies_AC} that $A$ is isomorphic to an inductive limit of finitely generated, projective \ca{s} with generator rank at most $\gen(A)+1$.

The implications `(5)$\Rightarrow$(3)', `(5)$\Rightarrow$(7)', and `(7)$\Rightarrow$(4)' are clear.
The implication `(3)$\Rightarrow$(4)' follows from \autoref{prop:S04:lma:proj_map_null-homotopic}.
\end{proof}

\begin{cor}
\label{prop:S04:contr_AP}
Every separable, contractible \ca{} is an inductive limit of separable, projective \ca{s}.
\end{cor}

\begin{thm}
\label{prop:S04:Permanence_trivSh}
The class of separable \ca{s} with trivial shape is closed under:
\begin{enumerate}
\item
countable direct sums;
\item
sequential inductive limits;
\item
approximation by sub-\ca{s} (that is, `likeness' as in \autoref{defn:S03:P-like});
\item
taking maximal tensor products with \emph{any} other separable \ca{}.
\end{enumerate}
\end{thm}
\begin{proof}
(1).
Let $(A_k)_{k\in\NN}$ be a family of \ca{s} with trivial shape.
By statement (7) of \autoref{prop:S04:TFAE-trivSh}, each $A_k$ is an inductive limit of contractible \ca{s}.
Note that countable direct sums of contractible \ca{s} are again contractible.
Hence, $\bigoplus_kA_k$ is an inductive limit of contractible \ca{s} and thus has trivial shape by \autoref{prop:S04:TFAE-trivSh}.

(2).
Assume that $A\cong\varinjlim A_k$ with each $A_k$ having trivial shape.
By \autoref{prop:S04:TFAE-trivSh}, each $A_k$ is an inductive limit of projective \ca{s}.
From \autoref{prop:S03:AAC_implies_AC} we deduce that $A$ is an inductive limit of projective \ca{s}, and so it has trivial shape using \autoref{prop:S04:TFAE-trivSh} again.

(3).
Assume that a \ca{} $A$ is approximated by sub-\ca{s} $A_i\subseteq A$ that have trivial shape.
By \autoref{prop:S04:TFAE-trivSh}, each $A_i$ is an inductive limit of projective \ca{s}.
This means that $A$ is $A\mathcal{P}$-like for the class $\mathcal{P}$ of projective \ca{s}.
It follows from \autoref{prop:S03:AC-like_implies_AC} that $A$ is an $A\mathcal{P}$-algebra, that is, an inductive limit of projective \ca{s}, and so $A$ has trivial shape by \autoref{prop:S04:TFAE-trivSh}.

(4).
Let $A$ be a \ca{} with trivial shape, and let $B$ be any (separable) \ca{}.
By statement (6) of \autoref{prop:S04:TFAE-trivSh}, we can write $A$ as an inductive limit of cones $\cone{A_k}=C_0((0,1])\otimes A_k$.
As noted by Blackadar, \cite[II.9.6.5, p.188]{Bla06OpAlgs}, maximal tensor products commute with arbitrary inductive limits (while minimal tensor products only commute with inductive limits with injective connecting morphisms).
Thus, $A\otimesMax B$ is an inductive limit of the cones $\cone{A_k}\otimesMax B=C_0((0,1])\otimes A_k\otimesMax B=\cone{(A_k\otimesMax B)}$.
Using statement (6) of \autoref{prop:S04:TFAE-trivSh} again, we deduce that $A\otimesMax B$ has trivial shape.
\end{proof}

Using the notation from \autoref{pargr:S03:Approximation} and \autoref{defn:S03:P-like}, the following corollary says that a contractible-like \ca{} has trivial shape and is approximately contractible.

\begin{cor}
\label{prop:S04:contr-like_implies_trivSh}
Let $A$ be a separable \ca{} that is approximated by contractible sub-\ca{s}.
Then $A$ has trivial shape.
Moreover, $A$ is an inductive limit of contractible \ca{s}.
\end{cor}

\begin{prp}
\label{prop:S04:change_connecting_mor_surj}
Let $(A_k,\gamma_k)$ be an inductive system of separable \ca{s}.
Then there exists an inductive system $(B_k,\delta_k)$ with surjective connecting morphisms and such that $\varinjlim A_k\cong\varinjlim B_k$.
Moreover, we may assume that $B_k=A_k\ast\mathcal{F}_\infty$ (the free product), where
\begin{align*}
\mathcal{F}_\infty:=C^\ast(x_1,x_2,\ldots : \|x_i\|\leq 1)
\end{align*}
is the universal \ca{} generated by a countable number of contractive generators.
If $A_k$ is (semi-)projective, then so is $A_k\ast \mathcal{F}_\infty$.
\end{prp}
\begin{proof}
Since the algebras $A_k$ are separable, there exists a surjective morphism $\varphi_k\colon\mathcal{F}_\infty\to A_k$ for each $k$.
Consider the universal \ca{}
\[
\mathcal{G}:=C^\ast(x_{i,j} : i,j\in\NN, \|x_{i,j}\|\leq 1).
\]
The only difference between $\mathcal{G}$ and $\mathcal{F}_\infty$ is in the enumeration of generators, and therefore $\mathcal{G}\cong\mathcal{F}_\infty$.

Set $B_k:=A_k\ast\mathcal{G}$ and define a morphism $\psi_k\colon\mathcal{G}\to B_{k+1}$ via $\psi_k(x_{1,j}):=\varphi_{k+1}(x_j)$, and $\psi_k(x_{i,j}):=x_{i-1,j}$ if $i\geq 2$.
Define a morphism $\delta_k\colon B_k\to B_{k+1}$ as $\delta_k:=\gamma_k\ast\psi_k$.
It is easy to check that $\delta_k$ is surjective.
For each $i$, the elements $x_{i,1},x_{i,2},\ldots\in\mathcal{G}$ generate a copy of $\mathcal{F}_\infty$.
In this way, we may think of $\mathcal{G}$ as a countable free product of copies of $\mathcal{F}_\infty$.
Then, the map $\delta_k$ looks as follows:
\begin{center}
\makebox{
\xymatrix{
B_k \ar@{}[r]|{:=} \ar[d]_{\delta_k}
    & A_k \ar@{}[r]|{\ast} \ar[d]_{\gamma_k}
    & \mathcal{F}_\infty \ar@{}[r]|{\ast} \ar[dl]_{\varphi_k}
    & \mathcal{F}_\infty \ar@{}[r]|{\ast} \ar[dl]_{\cong}
    & \mathcal{F}_\infty \ar@{}[r]|{\ast} \ar[dl]_{\cong}
    & \ldots \\
B_{k+1} \ar@{}[r]|{:=}
    & A_{k+1} \ar@{}[r]|{\ast}
    & \mathcal{F}_\infty \ar@{}[r]|{\ast}
    & \mathcal{F}_\infty \ar@{}[r]|{\ast}
    & \mathcal{F}_\infty \ar@{}[r]|{\ast}
    & \ldots \\
}}
\end{center}

The natural inclusions $\iota_k\colon A_k\to B_k$ intertwine the connecting morphisms $\gamma_k$ and $\delta_k$, that is, $\delta_k\circ\iota_k=\iota_{k+1}\circ\gamma_k$.
Thus, the morphisms $\iota_k$ define a natural morphism $\iota\colon A=\varinjlim A_k\to B:=\varinjlim B_k$.
Since each $\iota_k$ is injective, so is $\iota$.

To show that $\iota$ is surjective, let $b\in B$ and $\varepsilon>0$.
We need to find $a\in A$ with $b=_\varepsilon\iota(a)$.
First, we choose and index $k$ and $b'\in B_k$ such that $\delta_{\infty,k}(b')=_{\varepsilon/2}b$.
By definition, $B_k=A_k\ast\mathcal{G}$.
This implies that every element of $B_k$ can be approximated by finite polynomials involving the elements of $A$ and the generators $x_{i,j}$.
Actually, we only need that $b'$ is approximated up to $\varepsilon/2$ by an element $b''$ in the sub-\ca{} $A_k\ast C^\ast(x_{i,j} : i\in\{1,2,\ldots,l\},j\in\NN, \|x_{i,j}\|\leq 1)$.
Note that $\delta_{k+l,k}(b'')$ lies in the image of $\iota_{k+l}$, say $\delta_{k+l,k}(b'')=\iota_{k+l}(x)$ for $x\in A_{k+l}$.
Then $a=\gamma_{\infty,k+l}(x)\in A$ satisfies $b=_\varepsilon\iota(a)$, which completes the proof of surjectivity.

Note that $\mathcal{F}_\infty$ is projective.
It follows from \cite[Proposition~2.6, 2.31]{Bla85ShapeThy} that $A_k\ast \mathcal{F}_\infty$ is \mbox{(semi-)}\-projective whenever $A_k$ is.
\end{proof}

\begin{cor}
\label{prop:S04:trivSh_surj_limit_proj}
If a separable \ca{} has trivial shape, then it is an inductive limit of projective \ca{} with surjective connecting morphisms.
\end{cor}

\begin{rmk}
\label{pargr:S04:exmpl_Dadarlat}
Dadarlat, \cite{Dad14StblyContrNotContr}, gives an example of a pointed, compact, Hausdorff space $X$ such that the commutative \ca{} $A:=C_0(X)$ is stably contractible (that is, $A\otimes\KK$ is contractible, and in particular has trivial shape), while $A$ is not contractible.
In fact, $X$ is a two-dimensional CW-complex with non-trivial fundamental group, so that $(X,x_0)$ does not have trivial shape (in the pointed, commutative category).
It follows from \cite[Proposition~2.9]{Bla85ShapeThy} that $A$ also does not have trivial shape (as a \ca{}).

Thus, while $A\otimes\KK$ has trivial shape, the full hereditary sub-\ca{} $A\subseteq A\otimes\KK$ does not.
This shows that trivial shape does not pass to full hereditary sub-\ca{s}.
From this we may deduce the following result.
\end{rmk}

\begin{prp}
\label{prop:S04:Projectivity_not_hereditary}
Projectivity does not pass to full hereditary sub-\ca{s}.
\end{prp}
\begin{proof}
Let $A$ be Dadarlat's example of a \ca{} with $A\otimes\KK\simeq 0$ while $A\nsim_{Sh}0$, see \cite{Dad14StblyContrNotContr} and \autoref{pargr:S04:exmpl_Dadarlat}.
By \autoref{prop:S04:trivSh_surj_limit_proj}, $A\otimes\KK$ is an inductive limit of projective \ca{} $P_k$ with surjective connecting morphisms $\gamma_k\colon P_k\to P_{k+1}$.
Consider the pre-images $Q_k:=\gamma_{\infty,k}^{-1}(A)\subseteq P_k$.
Since $A\subseteq A\otimes\KK$ is a full hereditary sub-\ca{}, so is $Q_k\subseteq P_k$.

Note that $A\cong\varinjlim Q_k$.
If all algebras $Q_k$ are projective, then $A$ has trivial shape by \autoref{prop:S04:TFAE-trivSh}.
Since this is not the case, some algebras $Q_k$ are not projective.
\end{proof}

\begin{rmk}
It was shown by Eilers and Katsura, \cite[Example~7.9]{EilKat17SemiprojGraphCa}, that also semiprojectivity does not pass to full hereditary sub-\ca{s}.
\end{rmk}

\section{Relations among the classes of (weakly) \mbox{(semi-)}\-projective \texorpdfstring{$C^*$-algebras}{C*-algebras}}
\label{sect:S05:relations_classes}

In this section we will study the relation among the four classes of (weakly) semiprojective \ca{s} and (weakly) projective \ca{s}.
As it turns out, the situation is completely analogous to the commutative setting.

\begin{lma}
\label{prop:S05:lma:SP_htpyDom_by_P_implies_P}
Let $A$ be a \ca{}, let $P$ be a projective \ca{}, and let $\alpha\colon A\to P$ and $\beta\colon P\to A$ be morphisms with $\beta\circ\alpha=\id_A$.
Then $A$ is projective.
\end{lma}
\begin{proof} 
Let $B$ be any \ca{}, let $J\lhd B$ be an ideal, and let $\varphi\colon A\to B/J$ be a morphism.
We need to find a lift $\psi\colon A\to B$.
The situation and the maps to be constructed are shown in the following commutative diagram.
\begin{center}
\makebox{
\xymatrix{
    & & & B \ar@{->>}[d]^{\pi} \\
    A \ar[r]_{\alpha} \ar@/_1.5pc/[rr]_{\id_A}
        & P \ar[r]_{\beta} \ar@{..>}[urr]^{\omega}
        & A \ar[r]_{\varphi}
        & B/J \\
}}
\end{center}
Since $P$ is projective, there exists a morphism $\omega\colon P\to B$ that lifts $\varphi\circ\beta\colon P\to B/J$, that is, such that $\pi\circ\omega=\varphi\circ\beta$.
Set $\psi:=\omega\circ\alpha\colon A\to B$.
Then $\pi\circ\psi=\pi\circ\omega\circ\alpha=\varphi\circ\beta\circ\alpha=\varphi\circ\id_A$.
\end{proof}

\begin{thm}
\label{prop:S05:SP-C_implies_P}
Let $A$ be a separable, semiprojective \ca{} of trivial shape.
Then $A$ is projective.
\end{thm}
\begin{proof} 
By \autoref{prop:S04:trivSh_surj_limit_proj}, $A$ is an inductive limit of projective \ca{s} $P_k$ with surjective connecting morphisms $\gamma_k\colon P_k\to P_{k+1}$.
The semiprojectivity of $A$ gives an index $k$ and a lift $\alpha\colon A\to P_k$ such that $\gamma_{\infty,k}\circ\alpha=\id_A$.
It follows from \autoref{prop:S05:lma:SP_htpyDom_by_P_implies_P} that $A$ is projective.
\end{proof}

\begin{cor}
\label{prop:S05:lma:SP_trivSh_implies_C}
Let $A$ be a separable, semiprojective \ca{} of trivial shape.
Then $A$ is contractible.
\end{cor}

Loring, \cite[Lemma~5.5]{Lor12WklyProj}, shows that for a weakly projective \ca{} $A$ and a semi\-projective \ca{} $D$ the set $[D,A]$ of homotopy classes of morphisms from $D$ to $A$ is trivial.
A variant of this proof shows statement (2) in \autoref{prop:S04:TFAE-trivSh}, so that we get the following:

\begin{prp}
\label{prop:S05:Lor:wP_implies_trivSh}
Every separable, weakly projective \ca{} has trivial shape.
\end{prp}

\autoref{prop:S05:Lor:wP_implies_trivSh} shows that a weakly projective \ca{} is weakly semiprojective and has trivial shape.
We will now show that the converse is also true.

\begin{thm}
\label{prop:S05:wSP-trivSh_implies_wP}
Let $A$ be a separable, weakly semiprojective \ca{} of trivial shape.
Then $A$ is weakly projective.
\end{thm}
\begin{proof}
Let $B$ be a \ca{}, let $J\lhd B$ be an ideal, and let $\pi\colon B\to B/J$ be the quotient morphism.
Further, let $\varphi\colon A\to B/J$ be a morphism, let $\varepsilon>0$, and let $F\ssubset A$.
We need to find a lift $\psi\colon A\to B$ such that $\pi\circ\psi=_\varepsilon^F\varphi$.
The situation and the maps to be constructed are shown in the following diagram.
\begin{center}
\makebox{
\xymatrix{
& & B \ar[d]^{\pi}
\\
P_k \ar[r]_{\gamma_{\infty,k}} \ar@{..>}[urr]^{\beta}
& A \ar@{..>}[ur]_{\psi} \ar[r]_{\varphi} \ar@/^1.5pc/[l]^{\alpha}
& B/J
}}
\end{center}

From \autoref{prop:S04:TFAE-trivSh} we get an inductive system $(P_k,\gamma_k)$ of projective \ca{s} $P_k$ with inductive limit $A$.
Considering the identity morphism $\id_A\colon A\to A\cong\varinjlim P_k$ we get from \autoref{prop:S02:WSP_lifting_limits} an index $k$ and a morphism $\alpha\colon A\to P_k$ such that $\gamma_{\infty,k}\circ\alpha=_\varepsilon^F\id_A$.
Consider the morphism $\varphi\circ\gamma_{\infty,k}\colon P_k\to B/J$.
The projectivity of $P_k$ gives us a lift $\beta\colon P_k\to B$ such that $\pi\circ\beta=\varphi\circ\gamma_{\infty,k}$.

Set $\psi:=\beta\circ\alpha$.
Then $\pi\circ\psi=\pi\circ\beta\circ\alpha=\varphi\circ\gamma_{\infty,k}\circ\alpha=_\varepsilon^F\varphi$, as desired.
\end{proof}

We summarize the results as follows:

\begin{thm}
\label{prop:S05:wP_iff_wSP_trivSh}
Let $A$ be a separable \ca{}.
Then the following are equivalent:
\begin{enumerate}
\item
$A$ is (weakly) projective.
\item
$A$ is (weakly) semiprojective and has trivial shape.
\end{enumerate}
\end{thm}

The following conclusion confirms a conjecture of Loring.

\begin{cor}
\label{prop:S05:P_iff_SP-C}
Let $A$ be a separable \ca{}.
Then the following are equivalent:
\begin{enumerate}
\item
$A$ is projective.
\item
$A$ is semiprojective and weakly projective.
\item
$A$ is semiprojective and contractible.
\item
$A$ is semiprojective and has trivial shape.
\end{enumerate}
\end{cor}

\begin{pgr}
We note that the results in this section, in particular \autoref{prop:S05:wP_iff_wSP_trivSh}, are in exact analogy with results in shape theory for spaces, as shown in the table below.

A (weakly) projective \ca{} is the non-commutative analog of an (approximate) absolute retract, and a (weakly) semiprojective \ca{} is the non-commutative analog of an (approximate) absolute neighborhood retract.
The analogies are shown in the table below.
We refer the reader to \cite[2.1, 2.2, 2.3]{SoeThi12CharCommutSP} and the references therein for definitions and further discussion.

We have the following analogy of notions:
\\

\newlength{\bulletIdent}
\settowidth{\bulletIdent}{$\bullet\:$}

\noindent
\begin{tabular}{p{.5\textwidth}|p{.4\textwidth}}
commutative world
& noncommutative world \\
(for a compact, metric space $X$):
& (for a separable \ca{} $A$): \\[0.15cm]
\hline\\
$\bullet\:$ $X$ is an absolute retract (AR)
&  $\bullet\:$ $A$ is projective (P) \\[0.2cm]
$\bullet\:$ $X$ is an approximate absolute retract (AAR)
&  $\bullet\:$ $A$ is weakly projective (WP) \\[0.2cm]
$\bullet\:$ $X$ is an absolute neighborhood retract (ANR)
&  $\bullet\:$ $A$ is semiprojective (SP) \\[0.2cm]
$\bullet\:$ $X$ is an approximative absolute neigh\-bor\-hood retract (AANR)
&  $\bullet\:$ $A$ is weakly semiprojective (WSP) \\[0.2cm]
\end{tabular}
\\

We have the following analogy of results:
\\

\noindent
\begin{tabular}{p{.5\textwidth}|p{.4\textwidth}}
commutative world
& noncommutative world \\
(for a compact, metric space $X$):
& (for a separable \ca{} $A$): \\[0.15cm]
\hline\\[0.1cm]

$\bullet\:$ $X$ is AR $\:\:\Leftrightarrow\:\:$ $X$ is ANR and $X\simeq\pt$
& $\bullet\:$ $A$ is P $\:\:\Leftrightarrow\:\:$ $A$ is SP and $A\simeq 0$ \\
(see \cite[IV.9.1]{Bor67ThyRetracts})
& (see \autoref{prop:S05:SP-C_implies_P}) \\[0.2cm]

$\bullet\:$ $X$ is AAR $\:\:\Leftrightarrow\:\:$ $X$ is AANR and $X\sim_{Sh}\pt$ 
(see \cite{Gmu71ApproxRetr} and \cite{Bog75ApproxFundRetr})
& $\bullet\:$ $A$ is WP $\:\:\Leftrightarrow\:\:$ $A$ is WSP and $A\sim_{Sh}0$ \\
& (see \cite{Lor12WklyProj} and \autoref{prop:S05:wSP-trivSh_implies_wP})\\[0.2cm]

$\bullet\:$ if $X$ is ANR, then:
& $\bullet\:$ if $A$ is SP, then:  \\
$X\sim_{Sh}\pt$ $\Leftrightarrow$ $X\simeq\pt$ \: (see \cite{Bor67ThyRetracts})
& $A\sim_{Sh}0$ $\Leftrightarrow$ $A\simeq 0$ \: (see \autoref{prop:S05:lma:SP_trivSh_implies_C}) \\
\end{tabular}
\end{pgr}

\section{Questions}

\begin{qst}
\label{quest:Proj_gen}
Assume that $A$ is a \ca{} with trivial shape.
Is $A$ an inductive limit, $\varinjlim A_k$, with surjective connecting morphisms of projective \ca{s} $A_k$ with $\gen(A_k)\leq\gen(A)+1$?
\end{qst}

The result of Loring and Shulmann, \cite[Theorem~7.4]{LorShu12NCSemialgLifting}, see \autoref{prop:S04:LorShu:Cone_as_limit}, shows that \autoref{quest:Proj_gen} has a positive answer for cones.
Furthermore, it follows from \autoref{prop:S04:TFAE-trivSh} that $A$ is an inductive limit, $\varinjlim A_k$, of projective \ca{s} $A_k$ with $\gen(A_k)\leq\gen(A)+1$, but the connecting morphisms may not be surjective.
Using \autoref{prop:S04:change_connecting_mor_surj}, we can always arrange that the connecting morphisms are surjective, but the approximating algebras are replaced by $A_k\ast\mathcal{F}_\infty$, which have $\gen(A_k\ast\mathcal{F}_\infty)=\infty$.
\\

Let us say that a \ca{} $A$ has property $(\ast)$ if $[D,A]$ is trivial for every semiprojective \ca{} $D$.
This means that for each (fixed) semiprojective $D$, all morphisms from $D$ to $A$ are homotopic.
Every \ca{} of trivial shape has property $(\ast)$.

\begin{qst}
Assume that $A$ is a \ca{} with property $(\ast)$.
Does $A$ have trivial shape?
\end{qst}

If $A$ is an inductive limit of semiprojective \ca{s}, then property $(\ast)$ for $A$ implies that $A$ has trivial shape.
As mentioned in \autoref{quest:S01:SSS}, see \cite[4.4]{Bla85ShapeThy}, it is however an open question whether every \ca{} is an inductive limit of semiprojective \ca{s}.

\section*{Acknowledgments}

I thank Eduard Ortega and Mikael R{\o}rdam for their valuable comments, and especially for their careful reading of all the technical details.
I thank Tatiana Shulman and Leonel Robert for discussions and feedback on this paper.
I thank George Elliott for interesting discussions on approximate intertwinings.


\providecommand{\bysame}{\leavevmode\hbox to3em{\hrulefill}\thinspace}
\providecommand{\noopsort}[1]{}
\providecommand{\mr}[1]{\href{http://www.ams.org/mathscinet-getitem?mr=#1}{MR~#1}}
\providecommand{\zbl}[1]{\href{http://www.zentralblatt-math.org/zmath/en/search/?q=an:#1}{Zbl~#1}}
\providecommand{\jfm}[1]{\href{http://www.emis.de/cgi-bin/JFM-item?#1}{JFM~#1}}
\providecommand{\arxiv}[1]{\href{http://www.arxiv.org/abs/#1}{arXiv~#1}}
\providecommand{\doi}[1]{\url{http://dx.doi.org/#1}}
\providecommand{\MR}{\relax\ifhmode\unskip\space\fi MR }
\providecommand{\MRhref}[2]{%
  \href{http://www.ams.org/mathscinet-getitem?mr=#1}{#2}
}
\providecommand{\href}[2]{#2}

\end{document}